\documentclass[12pt, reqno]{amsart}

\usepackage{graphicx}
\usepackage{subfigure}
\usepackage{eepic}
\usepackage{xcolor}
\selectcolormodel{gray}
\usepackage{tikz}
\usepackage{amsmath}
\usepackage{a4wide}
\usepackage[utf8]{inputenc}
\usepackage{amssymb}
\usepackage{amsopn}
\usepackage{epsfig}
\usepackage{amsfonts}
\usepackage{latexsym}
\usepackage{amsthm}
\usepackage{enumerate}
\usepackage[UKenglish]{babel}
\usepackage{verbatim}
\usepackage{color}

\newtheorem{thm}{Theorem}[section]
\newtheorem{lma}[thm]{Lemma}
\newtheorem{ass}[thm]{Assumption}

\newtheorem{prop}[thm]{Proposition}

\newtheorem{exam}[thm]{Example}

\newcommand{\R}{\mathbb{R}}
\newcommand{\tr}{\textnormal{tr}}

\newcommand{\N}{\mathbb{N}}

\newcommand{\C}{\mathbb{C}}

\providecommand{\norm}[1]{\lVert#1\rVert}

\renewcommand{\geq}{\geqslant}
\renewcommand{\leq}{\leqslant}
\renewcommand{\epsilon}{\varepsilon}

\renewcommand{\P}{\mathbb{P}}

\renewcommand{\l}{\text{loc}}
\renewcommand{\i}{\mathbf{i}}

\newcommand{\id}{\textnormal{Id}}

\renewcommand{\geq}{\geqslant}
\renewcommand{\leq}{\leqslant}

\providecommand{\A}{\mathcal{A}}
\renewcommand{\i}{\mathbf{i}}
\providecommand{\I}{\mathcal{I}}

\renewcommand{\u}{\mathbf{u}}

\renewcommand{\l}{\mathcal{L}}
\providecommand{\x}{\mathbf{x}}
\providecommand{\y}{\mathbf{y}}
\providecommand{\p}{\mathbf{p}}

\providecommand{\i}{\mathbf{i}}

\providecommand{\norm}[1]{\lVert#1\rVert}

\providecommand{\one}{\mathbf{1}}

\renewcommand{\c}{\mathcal{C}_a}

\DeclareMathOperator\Arg{Arg}

\begin{document}

\title[]{Effective estimates on the top Lyapunov exponents for random matrix products}

\author{Natalia Jurga and Ian Morris}
\address{Natalia Jurga: Department of Mathematics, University of Surrey, Guildford, GU2 7XH, UK}
\email{N.Jurga@surrey.ac.uk}
\address{Ian Morris: Department of Mathematics, University of Surrey, Guildford, GU2 7XH, UK}
\email{I.Morris@surrey.ac.uk }

\date{\today}

\subjclass[2010]{}

\begin{abstract}
We study the top Lyapunov exponents of random products of positive $2 \times 2$ matrices and obtain an efficient algorithm for its computation. As in the earlier work of Pollicott \cite{pol}, the algorithm is based on the Fredholm theory of determinants of trace-class linear operators. In this article we obtain a simpler expression for the approximations which only require calculation of the eigenvalues of finite matrix products and not the eigenvectors. Moreover, we obtain effective bounds on the error term in terms of two explicit constants: a constant which describes how far the set of matrices are from all being column stochastic, and a constant which measures the minimal amount of projective contraction of the positive quadrant under the action of the matrices.

\end{abstract}

\maketitle

\section{Introduction}\label{intro}

Let $\A=\{A_1, \ldots, A_k\}$ be a finite set of invertible $d \times d$ real matrices. Let $\I=\{1, \ldots, k\}$, $\Sigma=\I^{\N}$, $\p=(p_1, \ldots, p_k)$ be a probability vector and $\mu_{\p}$ be the Bernoulli measure on $\Sigma$ associated to $\p$. The \emph{Lyapunov exponent} of $\mu_{\p}$ is given by the limit
$$\Lambda(\A,\p)=\Lambda= \lim_{n \to \infty} \frac{1}{n} \int \log \norm{A_{i_1} \cdots A_{i_n}} \textup{d}\mu_{\p}(\i)$$
where $\i=(i_n)_{n \in \N} \in \Sigma$ and $\norm{\cdot}$ denotes any matrix norm. By the sub-additive ergodic theorem, for $\mu_{\p}$ almost every $\i \in \Sigma$,
$$\Lambda= \lim_{n \to \infty} \frac{1}{n} \log \norm{A_{i_1} \cdots A_{i_n}},$$
a result which was first established by Furstenberg and Kesten in \cite{fk}. The precise estimation of the top Lyapunov exponent -- which is to say, the evaluation of the above limit -- has been a problem of noted interest and difficulty from its first appearance to the present day (see for example \cite{bai,cpv,ki,ma,pol,tsbl}).

In this article we will be interested in obtaining \emph{effective} estimates on the Lyapunov exponent $\Lambda$. We make the following assumption on our set of matrices. 

\begin{ass}
$\{A_1, \ldots, A_k\}$ is a finite set of $2 \times 2$ positive invertible matrices. We denote the entries of $A_i$ by
$$A_i= \begin{pmatrix} a_i & b_i \\ c_i& d_i \end{pmatrix}.$$
We also assume that not all of the matrices are column stochastic, meaning that
\begin{eqnarray}
C_1=\max_{1 \leq i \leq k} \left\{a_i+c_i, b_i+d_i, \frac{1}{a_i+c_i}, \frac{1}{b_i+d_i}\right\} >1.\label{c1} \end{eqnarray}
\label{ass}
\end{ass}

The positivity assumption in Assumption \ref{ass} ensures that the matrices map the positive quadrant strictly inside itself.  We lose no generality by assuming (\ref{c1}) since whenever $C_1=1$ it follows that $\Lambda=0$, see section \ref{op}.

In \cite{pol} Pollicott constructed an algorithm for estimating the Lyapunov exponent of a finite set of positive matrices by studying $k$ analytic families of operators $\{L_{1,t}\}_{t \in \C}, \ldots, \{L_{k,t}\}_{t \in \C}$, which we shall denote collectively as $\{L_{j,t}\}$. He showed that $\Lambda= \sum_{j=1}^k p_j \frac{\textup{d}}{\textup{d}t} \lambda_1(L_{j,t}) \bigr|_{t=0}$ where $\lambda_1(L_{j,t})$ denotes the top eigenvalue of the operator $L_{j,t}$, and thus reduced the problem to extracting an estimate on the eigenvalues $\lambda_1(L_{j,t})$ for $t \in \C$ close to 0. In the same way that the eigenvalues of a matrix can be extracted by studying the roots of the characteristic polynomial of a matrix, Pollicott applied Grothendieck's theory of determinants on Banach spaces to show that the eigenvalues of $L_{j,t}$ appeared as the roots of the Fredholm determinant $\det(\id-z L_{j,t})$. He then showed that the determinant function, being an entire function of $z$, had coefficients which could be computed explicitly by considering finite matrix products, leading to the following non-effective approximation of $\Lambda$. This approach to the spectra of transfer operators was pioneered by D. Ruelle in \cite{ruelle}.

With slight abuse of notation, for any matrix $A=A_{i_1} \cdots A_{i_n} \in \A^n$ we define $p_{A}=p_{i_1} \cdots p_{i_n}$. 

\begin{thm}[Pollicott \cite{pol}]
For each $n \geq 1$ and $i \in \{1, \ldots, k\}$ define
$$\tau_{n,i}= \sum_{A \in \A^n} p_{A}\left(1-\frac{\lambda_2(A)}{\lambda_1(A)}\right)^{-1}\sum_{j=0}^{n-1}\log\left(\frac{\norm{A_i v_{A,j}}}{\norm{v_{A,j}}}\right)$$
and
$$t_n=\sum_{A \in \A^n} p_{A}\left(1-\frac{\lambda_2(A)}{\lambda_1(A)}\right)^{-1}$$
where for $A=A_{i_1} \cdots A_{i_n}$, $v_{A,j}$ corresponds to the leading eigenvector of $A_{i_{j+1}} \cdots A_{i_n} A_{i_1} \cdots A_{i_j}$. Define
$$\alpha_{n,i}= \sum_{l=1}^n \frac{(-1)^l}{l!} \sum_{n_1+ \ldots +n_l=n} \sum_{j=1}^l \frac{\tau_{n_j,i}}{n_j} \prod_{1 \leq m \leq l, m \neq j} \frac{t_{n_m}}{n_m}$$
and
$$a_n= \sum_{l=1}^n \frac{(-1)^l}{l!} \sum_{n_1+ \ldots +n_l=n} \prod_{i=1}^l \frac{t_{n_i}}{n_i}.$$
Then 
$$\Lambda_N=\frac{\sum_{n=1}^N \sum_{i=1}^k \alpha_{n,i}}{\sum_{n=1}^N na_n}$$
satisfies $|\Lambda_N-\Lambda|=O(\exp(-\gamma N^2))$ for some $\gamma>0$. \label{thm pol}
\end{thm}

Here $\|\cdot\|$ denotes the Euclidean norm. The article \cite{pol} is one of a growing number of papers which exploit the strong spectral properties of transfer operators that preserve spaces of analytic functions and the Fredholm theory of determinants in order to approximate quantities which can be expressed in terms of the spectrum of some dynamical transfer operator; recent examples include approximations of the singularity dimension of certain self-affine iterated function systems, the Hausdorff dimension of dynamically defined Cantor sets and diffusion coefficients for expanding maps \cite{morris, pj, jpv, pv}. A key feature of this machinery is that it yields approximations with an error estimate that decays super-exponentially fast, as demonstrated in Theorem \ref{thm pol}. Note that the time to process $n$ steps of the algorithm is exponential in $n$ and therefore the error decreases super-polynomially fast in time. The estimates provided by Pollicott's algorithm in Theorem \ref{thm pol} are non-effective, although in \cite[\S 8]{pol} he informally sketches some ideas of how to make his estimates effective. In this article we set out to obtain effective estimates on the Lyapunov exponent $\Lambda$ by using similar ideas of trace-class operators and determinants. Our motivation is threefold: firstly, we obtain a formula for the top Lyapunov exponent which is significantly simpler than Pollicott's in that it requires only the computation of eigenvalues and not of eigenvectors; secondly, our estimates are effective as opposed to qualitative, and allow explicit rigorous bounds on the top Lyapunov exponent to be made; thirdly, by investigating more explicitly the relationship between the matrix entries and the error term we can examine more explicitly for which matrices this method quickly produces a useful approximation and for which matrices it does not.

\begin{thm} \label{main}
For each $n \geq 1$ define
$$\tau_{n}= \sum_{A \in \A^n} p_{A}\log \lambda_1(A)\left(1-\frac{\lambda_2(A)}{\lambda_1(A)}\right)^{-1}$$
and
$$t_n=\sum_{A \in \A^n} p_{A}\left(1-\frac{\lambda_2(A)}{\lambda_1(A)}\right)^{-1}$$
Define
$$\alpha_{n}= \sum_{l=1}^n \frac{(-1)^l}{l!} \sum_{n_1+ \ldots +n_l=n} \sum_{j=1}^l \frac{\tau_{n_j}}{n_j} \prod_{1 \leq m \leq l, m \neq j} \frac{t_{n_m}}{n_m}$$
and
$$a_n= \sum_{l=1}^n \frac{(-1)^l}{l!} \sum_{n_1+ \ldots +n_l=n} \prod_{i=1}^l \frac{t_{n_i}}{n_i}.$$
Then
$$\Lambda_N= \frac{\sum_{n=1}^N \alpha_n}{\sum_{n=1}^N na_n}$$
satisfies $|\Lambda-\Lambda_N| \leq C\exp(-\gamma N^2)$ for some explicit constants $\gamma, C$ that depend only on the matrices $\{A_1, \ldots, A_k\}$. 
\end{thm}

Observe that our approximation is simpler than the one given in Theorem \ref{thm pol} since it requires only the calculation of eigenvalues and not eigenvectors, as well as the fact that it only requires the computation of the expression $\tau_n$ rather than $\tau_{n,1}, \ldots, \tau_{n,k}$. Essentially these differences in the algorithms boil down to the fact that Pollicott characterises the top Lyapunov exponent as the integral $\int \sum_{i=1}^k p_i \log \left(\frac{\norm{A_i v}}{\norm{v}}\right) \textup{d}\nu(v)$ with respect to the Furstenberg measure $\nu$ on $\R\P^1$ and shows that for each $i=1, \ldots, k$, $\int \log \left(\frac{\norm{A_i v}}{\norm{v}}\right) \textup{d}\nu(v)=\frac{\textup{d}}{\textup{d}t} \lambda_1(\l_{i,t})|_{t=0}$ where $\{\l_{i,t}\}_{t \in \C}$ is some family of operators. Instead, we express the top Lyapunov exponent more directly as the derivative with respect to $t$ at 0 of
$$\lim_{n \to \infty} \left( \sum_{A \in \A^n} p_{A}\norm{A}^t  \right)^{\frac{1}{n}}$$
and approximate this by finding a \emph{single} family of operators $\{\l_t\}_{t \in \C}$ such that the derivative of the above expression at $t=0$ is given by $\frac{\textup{d}}{\textup{d}t} \lambda_1(\l_t)|_{t=0}$.

In order to state the explicit bound on the error $|\Lambda-\Lambda_N|$, for each matrix $A_i \in \A$, let $R_i$ denote the smallest constant for which
\begin{eqnarray}
\frac{1}{R_i} \leq \frac{a_i}{c_i} \leq R_i & \textnormal{and} & \frac{1}{R_i} \leq \frac{b_i}{d_i} \leq R_i\label{R} \end{eqnarray}
and define 
\begin{eqnarray} r_{\A}= r:= \max_{1 \leq i \leq k} \left\{\frac{R_i-1}{R_i+1}\right\}<1.\label{r} \end{eqnarray}
Thus $r$ measures the least amount of projective contraction of the positive quadrant under the action of a matrix belonging to $\A$.

The error $|\Lambda-\Lambda_N|$ can be bounded purely in terms of the constant $C_1$ defined in (\ref{c1}), which can be considered a measure of how far the set of matrices $\A$ is from being a set made up only of column stochastic matrices and the constant $r$. Although the error can be bounded purely in terms of these two constants, we also introduce a third constant which generally will optimise the bound on the error. The constant $\theta$ is defined by
\begin{eqnarray}
\theta:= \max_{1 \leq i \leq k} \left\{ \arcsin\left(\frac{|a_i+c_i-b_i-d_i|}{a_i+b_i+c_i+d_i}\right)\right\}< \frac{\pi}{2}. \label{theta}
\end{eqnarray}
For an upper bound on the error which does not depend on $\theta$ one can replace all instances of $\theta$ below by $\frac{\pi}{2}$. We will now give an explicit bound on $|\Lambda-\Lambda_N|$ in terms of the constants $C_1, r, \theta$.

Let  $C_0= \frac{1}{r\sqrt{1-r^2}}$, $C_2=\sqrt{(\log C_1)^2+\theta^2}$ and let $M \geq 2$ be large enough that $C_0r^{\frac{M+1}{2}}<1$. Let $N$ be sufficiently large that
$$\alpha_N^{+}=\sum_{n=N+1}^{\infty} n\frac{C_0^n r^{\frac{n(n+1)}{2}}}{\prod_{i=1}^n (1-r^i)} < |1-r|^{M-2}\prod_{n=M}^{\infty}(1-C_0r^{\frac{n+1}{2}})=\alpha^{-} .$$
Then
\begin{eqnarray}
|\Lambda-\Lambda_N| &\leq& \frac{\beta_N^{+}}{\alpha^{-}-\alpha_N^{+}} + \frac{\alpha_N^{+}\beta^{+}}{\alpha^{-}(\alpha^{-}-\alpha_N^{+})}\label{rbound} \end{eqnarray}
where
\begin{eqnarray*}
\alpha^{+}=\sum_{n=1}^{\infty} n\frac{C_0^n r^{\frac{n(n+1)}{2}}}{\prod_{i=1}^n (1-r^i)} \\
\beta^{+}=\sum_{n=0}^{\infty} \frac{neC_2C_0^n r^{\frac{n(n+1)}{2}}}{\prod_{i=1}^n (1-r^i)} \\
\beta_N^{+}=\sum_{n=N+1}^{\infty} \frac{neC_2C_0^n r^{\frac{n(n+1)}{2}}}{\prod_{i=1}^n (1-r^i)}. 
\end{eqnarray*}

For clarity we have presented an upper bound given purely in terms of $r$, $\theta$ and $C_1$ although the bound given by (\ref{rbound}) can actually be improved slightly by modifying the definition of $\alpha^{-}$: we present this improved bound in section \ref{summary}. The reason for this is that it is possible to replace $r$ by a smaller constant in the factor $|1-r|^{M-2}$ which appears in the definition of $\alpha^{-}$. This smaller constant is  the `weighted Birkhoff contraction coefficient', which can also be easily defined in terms of the matrix entries of $A_i$ (which we postpone till the next section) and which is always bounded above by $r$, and may coincide with $r$. However, if $r$ is small (meaning that each of the matrices in $\A$ contract the positive quadrant a lot) and $M=2$ then the factor $|1-r|^{M-2}$ does not even appear in the error bound. Indeed, unless $r$ is close to 1 (meaning that the positive quadrant isn't contracted much by the action of the matrices in $\A$) and therefore $M$ is very large, replacing $r$ by the weighted Birkhoff contraction coefficient is unlikely to have any significant effect on the bound. For certain sets of matrices, it is possible to considerably improve the upper bound on the errors $|\Lambda-\Lambda_N|$; see section \ref{eg} for details.

The principal qualitative information which can be gained from Theorem \ref{main} is that our method for estimating Lyapunov exponents is most effective when the constant $r$ is small, particularly when it is less than one half. 

The paper is organised as follows. In \S\ref{prel} we introduce the family of operators that will be studied and recap the relevant functional analytic tools which will be used. In \S \ref{sketch} we will devise the algorithm for computing the approximations $\Lambda_N$, which will be based on the spectral properties of the operators that were introduced in \S\ref{prel}. In \S\ref{est} we obtain effective estimates on the error $|\Lambda-\Lambda_N|$. In \S\ref{eg} we test the performance of our algorithm and the upper bounds on the errors on some examples, and in \S\ref{further} we discuss the prospects for higher-dimensional analogues of our results.

\section{Preliminaries}\label{prel}

\subsection{Projective action and Birkhoff contraction coefficient}

For the remainder of the paper we let $\norm{\cdot}$ denote the Euclidean norm. Let $\R \P^1$ denote the real projective space of dimension 1, that is, $\R\P^1=(\R^2 \setminus (0,0), \sim )$ where $\mathbf{v} \sim \mathbf{w}$ if $\mathbf{v}=\alpha \mathbf{w}$ for some $\alpha \in \R$, $\alpha \neq 0$. Let $\R\P^1_+$ denote the open set of positive directions in $\R\P^1$, with representative vectors $\Delta=\{(x, 1-x): x \in (0,1)\}$. Given $\mathbf{x} \in \R^2 \setminus (0,0)$ we will denote its direction by $\overline{\mathbf{x}} \in \R\P^1$.

Let $A= \begin{pmatrix} a&b \\c&d \end{pmatrix}$ be a positive invertible matrix. Then we can consider the action of $A$ on $\R\P^1$ denoted by $A \cdot \overline{\mathbf{x}}=\overline{A \mathbf{x}}$. We can also consider the corresponding action on the first co-ordinate of representative vectors. Define $\tilde{\phi}_A: (0,1) \to (0,1)$ by
\begin{eqnarray}
\tilde{\phi}_A(x)= \frac{(a-b)x+b}{(a+c-b-d)x+b+d}.
\label{identify}
\end{eqnarray}
Clearly $(\tilde{\phi}_A(x), 1-\tilde{\phi}_A(x)) \in \Delta$ and $\overline{A\begin{pmatrix} x\\1-x\end{pmatrix}}=\overline{(\tilde{\phi}_A(x), 1-\tilde{\phi}_A(x))}$. Therefore, $\tilde{\phi}_A(x)$ describes the action of $A$ on the first co-ordinate of vectors in $\Delta$. Notice that the denominator in (\ref{identify}) is
$$(a+c-b-d)x+b+d= \langle A\x,\u \rangle $$
where $\x=(x,1-x)$, $\u=(1,1)$ and $\langle, \rangle$ denotes the usual dot product 
$$\langle A\x,\u\rangle =A \begin{pmatrix}x \\1-x \end{pmatrix} \cdot \begin{pmatrix} 1\\1\end{pmatrix}.$$
We define $\tilde{w}_A:(0,1) \to \R$, $\tilde{w}_A(x)=(a+c-b-d)x+b+d$. In particular we have the following identity
\begin{eqnarray}
\begin{pmatrix} \tilde{\phi}_A(x) \\1-\tilde{\phi}_A(x) \end{pmatrix}= \frac{1}{\tilde{w}_A(x)} A \begin{pmatrix} x \\1-x \end{pmatrix}.
\label{identity}
\end{eqnarray}

In section \ref{op} we will consider complex extensions of $\tilde{w}_{A_i}$ and $\tilde{\phi}_{A_i}$ in order to construct a suitable operator on a space of complex valued functions that will aid us in studying the Lyapunov exponent. The following elementary but important fact will be key to extracting the Lyapunov exponent from the spectral data of the operators we construct.

\begin{lma}
Let $\u=(1,1)$. There exists a uniform constant $c$, that depends only on the set $\A$ such that for all $x \in (0,1)$, $\x=(x,1-x)$ and $A \in \bigcup_{n =1}^{\infty}\A^n$
\begin{eqnarray}
c^{-1}\norm{A} \leq \langle A\x,\u \rangle \leq c\norm{A}.\label{dot eqn} \end{eqnarray}
\label{dot norm}
\end{lma}

\begin{proof}
Fix arbitrary $n \in \N$ and let $A \in \A^n$. To verify the right hand side, notice that by the Cauchy-Schwarz inequality, for all $\mathbf{y} \in \R^2$,
$$|\langle \mathbf{y}, \mathbf{u}\rangle | \leq \sqrt{2} \norm{\mathbf{y}}$$
and therefore since $0<x <1$,
$$|\langle A\mathbf{x}, \mathbf{u} \rangle| \leq \sqrt{2}\norm{A\mathbf{x}} \leq \sqrt{2}\norm{A}.$$

To verify the left hand side we begin by claiming that there exist uniform constants $\epsilon , \delta>0$ which depend only on $\A$ such that
\begin{eqnarray}
|\langle A\mathbf{x}, \mathbf{u}\rangle | &\geq& \epsilon \norm{A \mathbf{x}} \label{epsilon} \\
\norm{A \mathbf{x}} &\geq & \delta \norm{A} \norm{\mathbf{x}} \label{delta}
\end{eqnarray}
which are independent of the choice of $\x$ and $A$. 

First, to see that (\ref{epsilon}) holds, observe that the unoriented angle $\alpha$ between $A\mathbf{x}$ and $\mathbf{u}$ satisfies $0 \leq \alpha \leq \frac{\pi}{4}$ and therefore
$$|\langle A\mathbf{x}, \mathbf{u}\rangle |= \norm{A\mathbf{x}}\norm{\mathbf{u}} \cos \alpha \geq \norm{A\mathbf{x}} \frac{\sqrt{2}}{2} \sqrt{2}=\norm{A\mathbf{x}}.$$
Next, to see that (\ref{delta}) holds, let $\mathcal{C} \subset \mathrm{Int}(\R_{+}^2) \cup \{0\}$ be a closed convex cone which is preserved by all of the matrices in $\A$ and all of the transposes of matrices in $\A$. Let $A \in \bigcup_{n=1}^{\infty} \A^n$ and note that $\mathcal{C}$ is preserved by $A A^{\mathrm{T}}$, where $A^{\mathrm{T}}$ denotes the transpose of $A$. By the Perron-Frobenius theorem, there exists $\mathbf{v} \in \mathcal{C}$ with $\norm{\mathbf{v}}=1$ such that $AA^{\mathrm{T}}\mathbf{v}=\lambda_1(AA^{\mathrm{T}})\mathbf{v}=\norm{A}^2 \mathbf{v}$. In particular, this implies that $\norm{A^{\mathrm{T}}\mathbf{v}}=\norm{A}$. There exists $0< \alpha'< \frac{\pi}{2}$ which only depends on $\mathcal{C}$ (and thus only on the set of matrices $\A$) such that the unoriented angle between $\mathbf{x}$ and $A^{\mathrm{T}}\mathbf{v}$ is at most $\alpha'$. Putting all of this together we get
$$\norm{A\mathbf{x}} \geq \langle A\mathbf{x}, \mathbf{v}\rangle= \langle \mathbf{x}, A^{\mathrm{T}}\mathbf{v}\rangle \geq \norm{x}\norm{A^{\mathrm{T}}\mathbf{v}} \cos \alpha'= \norm{A} \norm{x} \cos \alpha',$$
completing the proof of (\ref{delta}).
Therefore for all $x \in (0,1)$ and $\mathbf{x}=(x, 1-x)$,
\begin{eqnarray*}
|\langle A\mathbf{x}, \mathbf{u}\rangle| &\geq& \epsilon \norm{A \mathbf{x}} \\
&\geq& \epsilon \delta \norm{A} \norm{\mathbf{x}} \\
&\geq& \frac{\epsilon \delta}{\sqrt{2}} \norm{A}|\langle\mathbf{x}, \mathbf{u}\rangle|= \frac{\epsilon \delta}{\sqrt{2}} \norm{A}.
\end{eqnarray*}

\end{proof}

Let us return to the projective space $\R\P^1$ and the projective action of a positive matrix $A$. Let $\overline{\x}, \overline{\y} \in \R\P^1_+$ be positive directions. We can equip $\R\P_+^1$ with the \emph{Hilbert projective metric} by setting
$$h(\overline{\mathbf{x}}, \overline{\mathbf{y}})= \left|\log \frac{x_1 y_2}{x_2 y_1} \right|$$ 
where $(x_1, x_2)$ and $(y_1, y_2)$ are some representatives for the directions $\overline{\mathbf{x}}, \overline{\mathbf{y}}$. Note that $h(\overline{\mathbf{x}}, \overline{\mathbf{y}})$ is independent of the choice of representative vectors. Under the assumption that $A$ is positive, $A$ acts as a strict contraction of the Hilbert metric on $\R\P^1_+$. In particular by defining 
$$\psi(A)= \min\left\{\frac{ad}{bc}, \frac{bc}{ad}\right\}$$
and defining the \emph{Birkhoff contraction coefficient} as
\begin{eqnarray}\tau(A)=\frac{1-\psi(A)^{\frac{1}{2}}}{1+\psi(A)^{\frac{1}{2}}}<1 \label{birkhoff} \end{eqnarray}
then Birkhoff \cite{birkhoff} showed that for all positive directions $\overline{\mathbf{x}}, \overline{\mathbf{y}} \in \R\P_{+}^1$,
$$h(A \cdot \overline{\mathbf{x}}, A \cdot \overline{\mathbf{y}}) \leq \tau(A) h(\overline{\mathbf{x}}, \overline{\mathbf{y}}).$$

We define the \emph{weighted Birkhoff contraction coefficient} associated to our set of matrices $\A$ and our probability vector $\p$ as 
\begin{eqnarray}\tau(\A, \p)=\sum_{i=1}^k p_i \tau(A_i) . \label{wbirkhoff} \end{eqnarray}

Notice that $\tau(A_i) \leq \frac{R_i-1}{R_i+1}$ and $\tau(\A,\p)\leq r$ where $R_i$ and $r$ were defined in (\ref{R}) and (\ref{r}) respectively. In general these inequalities are strict, unless the column ratios of each column of every matrix all coincide.

\subsection{Hardy-Hilbert space} We recall some standard facts about Hardy-Hilbert spaces, see for instance \cite{shapiro}. Let $D$ be an open disc of radius $\rho$ centred at $c \in \C$. The Hardy-Hilbert space $H^2(D)$ consists of all functions $f$ which are holomorphic on $D$ and such that $\sup_{r< \rho} \int_0^1 |f(c+re^{2\pi i t})|^2 \textup{d}t< \infty$. The inner product on $H^2(D)$ is defined by 
$$\langle f, g \rangle_{H^2}= \int_0^1 f(c+re^{2\pi i t})\overline{g(c+re^{2\pi i t})} \textup{d}t$$
which is well-defined since any element of $H^2(D)$ extends as an $L^2$ function of the boundary $\partial D$. The norm of $f \in H^2(D)$ is then given as $\norm{f}_{H^2}=\langle f,f \rangle_{H^2}^{\frac{1}{2}}$. 

An alternative characterisation of $H^2(D)$ is given as the space of all functions $f$ which are holomorphic on $D$ which can be expressed in the form
$$f(z)= \sum_{n=0}^{\infty} \alpha_k(f) \frac{(z-c)^k}{\rho^k}$$
for some square-summable sequence of complex numbers $\{\alpha_k(f)\}_{k=0}^{\infty}$. We will primarily utilise this second characterisation of $H^2(D)$.

The norm of $f \in H^2(D)$ can be written alternatively as
$$\norm{f}_{H^2}= \left(\sum_{k=0}^{\infty} |\alpha_k(f)|^2\right)^{\frac{1}{2}}.$$

Suppose that $f$ is bounded and holomorphic on $D$. Then $f \in H^2(D)$ and $\norm{f}_{H^2} \leq \norm{f}_{\infty}$.

Throughout the rest of this paper we fix $D$ to be the disc of radius $\frac{1}{2}$ centred at $\frac{1}{2}$.

\subsection{Trace class operators, determinants and approximation numbers} Given a compact operator $L:H \to H$ on a Hilbert space $H$, its $n$th \emph{approximation number} is defined as
$$s_n(L)= \inf\{\norm{L-K} : \textnormal{rank}(K) \leq n-1\}.$$
The $n$th approximation number coincides with the n$th$ \emph{singular value}, which is the square root of the $n$th-largest eigenvalue of the operator $L^*L$; this equivalence is sometimes useful in calculations (such as in \cite{bj2}) but will not be needed in this article.

A bounded linear operator on a complex separable Hilbert space $H$ is called \emph{trace-class} if $|L|_{\text{tr}}:=\sum_{n=1}^{\infty} s_n(L)< \infty.$ It is easy to see that if $L_1,L_2: H \to H$ are bounded operators then $s_n(L_1L_2) \leq s_n(L_1) \norm{L_2}$, therefore $|L_1L_2|_{\text{tr}} \leq |L_1|_{\text{tr}} \norm{L_2}$. Whenever $L$ is trace-class then any of its iterates $L^k$ is also trace-class. Given a trace-class operator $L$, the \emph{trace} is defined as
$$\tr (L)= \sum_{n=1}^{\infty} \langle Le_n, e_n \rangle_H$$
where $\{e_n\}$ is any orthonormal basis and $\langle , \rangle_H$ is the inner product for the Hilbert space $H$. This is consistent with respect to the choice of basis (see e.g. \cite[\S 4 Theorem 6.1]{ggk2}). Therefore it is easy to see that $\tr\left( \sum_{n=1}^N L_n\right)=\sum_{n=1}^N \tr (L_n)$.

Given a compact operator $L$, we denote by $\{\lambda_n(L)\}_{n \in \N}$ the monotone decreasing sequence of non-zero eigenvalues of $L$, listed with algebraic multiplicity. Note that by \cite[\S 4 Theorem 3.1 and Corollary 3.4]{ggk2} we have the inequalities
\begin{eqnarray}
\prod_{i=1}^n |\lambda_i(L)| &\leq& \prod_{i=1}^n s_i(L) \qquad \forall n \in \N,
\label{product ineq} \\
\sum_{i=1}^{\infty} |\lambda_i(L) |&\leq& \sum_{i=1}^{\infty} s_i(L). \label{sum ineq}
\end{eqnarray}
If $L$ is trace-class then it is compact and its sequence of eigenvalues $\lambda_n(L)$ is absolutely summable. Moreover, by Lidskii's theorem \cite[\S 4 Theorem 6.1]{ggk2} the trace of $L$ is equal to
$$\textnormal{tr}(L)= \sum_{n=1}^{\infty} \lambda_n(L).$$

For a trace-class operator $L$, we define\footnote{The Fredholm determinant of a trace-class operator admits several descriptions and as such may be defined in various ultimately equivalent ways. We choose this formula as a definition largely for the sake of its accessibility. Alternative approaches may be found in e.g. \cite{bs,ggk2}.} the \emph{Fredholm determinant} of $L$ to be the quantity
\begin{eqnarray}
\det(\id -zL)= \prod_{n=0}^{\infty} (1-z\lambda_n(L))
\label{det}
\end{eqnarray}
which is an entire function of $z$ \cite[Theorem 3.3]{bs}, so in particular there exist $a_n \in \C$ such that
$$\det(\id-zL)= \sum_{n=0}^{\infty} a_nz^n.$$
Note that by (\ref{det}) the roots of $\det(\id -zL)$ are precisely the reciprocals of the eigenvalues of $L$, and the degree of each zero is given by the multiplicity of the corresponding eigenvalue. Moreover, each coefficient $a_n$ can be expressed in terms of the traces of $L^m$ for $1 \leq m \leq n$; we will show this for a specific operator in (\ref{trace coef2}). 

On the other hand, by finding the coefficient of $z^n$ in (\ref{det}) we see that 
$$a_n=(-1)^n \sum_{i_1< \ldots <i_n} \lambda_{i_1}(L) \ldots \lambda_{i_n}(L)$$
and moreover by \cite[Corollary VI.2.6]{ggk}
\begin{eqnarray}
|a_n| \leq \sum_{i_1< \ldots< i_n} |\lambda_{i_1}(L) \ldots \lambda_{i_n}(L)| \leq \sum_{i_1< \ldots < i_n} s_{i_1}(L) \ldots s_{i_n}(L).\label{sing coef}
\end{eqnarray}
We will see that effective estimates on the Lyapunov exponents will depend on obtaining effective upper bounds on the coefficients of the determinant $\det(\id-zL)$ for an appropriate trace-class operator $L$ and therefore, in view of (\ref{sing coef}) this will boil down to obtaining effective upper bounds on the approximation numbers of $L$.

\subsection{Analytic perturbation theory}

We say that a bounded linear operator $L$ on a Banach space has \emph{spectral gap} if $L=\lambda P+N$ where $P$ is a rank one projection (so $P^2=P$ and $\dim(\textnormal{Im}(P))=1$), $N$ is a bounded operator with spectral radius $\rho(N)< |\lambda|$ and $PN=NP=0$. $L$ does not need to be compact in order to have a spectral gap, however if the operator $L$ \emph{is} compact and has a simple leading eigenvalue\footnote{Throughout the paper we say that an eigenvalue is simple if it is \emph{algebraically simple}, that is, the eigenvalue has a one-dimensional generalised eigenspace.} and no other eigenvalues with the same absolute value, it has a spectral gap. 

We can use the standard techniques of perturbation theory \cite{kato} to relate the Lyapunov exponent of a set of matrices to the spectral properties of an appropriate operator. The following perturbation theorem is presented in a more general form in \cite[Theorem 3.8]{hennion}.

\begin{thm}[Analytic perturbation theorem] \label{apt} Let $\{L_t\}_{t \in \C}$ be a family of bounded linear operators on a Banach space such that $t \mapsto L_t$ is holomorphic and $L_0$ has spectral gap. Then there exists an open neighbourhood $U \subset \mathbb{C}$ of 0 for which $L_t$ has spectral gap for all $t \in U$. Moreover  there exist $\lambda(t), P_t, N_t$ which are holomorphic families on $U$ such that:
\begin{enumerate}
\item[(a)] $L_t=\lambda(t) P_t+N_t$,
\item[(b)] $N_t P_t=P_t N_t=0$
\item[(c)] $P_t$ is a bounded rank one projection and has the form $$P_t= \frac{1}{2\pi i} \int_{\gamma} (s \id -L_t)^{-1} \textup{d}s $$
for some small circle $\gamma$ around $\lambda$ which separates it from the rest of the spectrum of $L_0$,
\item[(d)] $\rho(N_t)< |\lambda(t)|- \epsilon$ for some $\epsilon>0$ which is independent of $t$.
\end{enumerate}
\end{thm}

\subsection{Transfer operator} \label{op} 

Recall that $A_i=\begin{pmatrix} a_i&b_i \\c_i& d_i \end{pmatrix}$. For each $i \in \{1, \ldots, k\}$ let $\phi_{A_i}: D \to \C$ denote the complex extension of $\tilde{\phi}_{A_i}$ to $D$ given by
\begin{eqnarray}\phi_{A_i}(z)= \frac{(a_i-b_i)z+b_i}{(a_i+c_i-b_i-d_i)z+b_i+d_i}.\label{phi} \end{eqnarray}
Note that this extension is well-defined since $\Re(z)>0$ and $\Re(1-z)>0$ for all $z \in D$ (where $\Re(z)$ denotes the real part of $z$) and therefore the real part of the denominator is positive; in particular the denominator does not vanish anywhere on $D$.

Given $A=A_{i_1}\cdots A_{i_n} \in \A^n$ we denote $\phi_A= \phi_{A_{i_1}} \circ \ldots \circ \phi_{A_{i_n}}$. It is easy to see that if $A \in \A^n$ is given by $A= \begin{pmatrix} a&b \\c&d \end{pmatrix}$ then $\phi_A(z)=\frac{(a-b)z+b}{(a+c-b-d)z+b+d}.$ 

Observe that for each $A \in \A^n$, $\phi_{A}$ has a unique fixed point. To see this, recall that by the Perron-Frobenius theorem $A$ has a positive eigenvalue $\lambda_1(A)$ and a corresponding positive eigenvector $(v_1, v_2)$. Therefore, putting $z_A=\frac{v_1}{v_1+v_2} \in (0,1)$ we see that $\phi_{A}(z_A)=z_A$. To see that the fixed point is unique,  define the holomorphic function $g_A: B(0,1) \to \C$ by $g_A(z)=2\phi_A(\frac{z+1}{2})-1$, where $B(0,1)$ denotes the open unit disk centred at 0. Observe that $y_A=2z_A-1$ is a fixed point of $g_A$. Also, since $\phi_A(D) \subset D$, $g_A$ preserves $B(0,1)$ and therefore by Schwarz's lemma, $y_A$ is the unique fixed point of $g_A$. In particular, $z_A$ is the unique fixed point of $\phi_A$.

We define the \emph{transfer operator} $\l_0: H^2(D) \to H^2(D)$  as
$$\l_0f= \sum_{i=1}^k p_i f\circ \phi_{A_i}.$$
It is easy to see that for each $i$, $\phi_{A_i}(D) \subset D$ since $\phi_{A_i}$ maps $D$ to a disk centred on the real axis whose boundary passes through the points $\frac{b_i}{b_i+d_i}$ and $\frac{a_i}{a_i+c_i}$.  Therefore since $\phi_{A_i}$ is a holomorphic self-map of $D$, by Littlewood's theorem \cite[page 11]{shapiro} it follows that $\l_0f \in H^2(D)$.

It is easy to see that $1$ is an eigenvalue of $\l_0$ for the eigenfunction $\one$. In fact it is a simple, maximal eigenvalue of $\l_0$.

\begin{prop}
1 is a simple maximal eigenvalue of $\l_0$, and is the only eigenvalue of modulus 1.
\label{simple}
\end{prop}

\begin{proof}
It is easy to see that 1 is an eigenvalue of $\l_0$ and that $\one$ is an eigenfunction for this eigenvalue. We begin by showing that it is a geometrically simple eigenvalue. Suppose that $f \in H^2(D)$ is a fixed point of $\l_0$ and that $f \neq 0$. We will show that $f$ must be a constant function. First, observe that
$$|f(z)|=|\l_0 f(z)| \leq \sum_{i=1}^k p_i|f\circ \phi_{A_i}(z)| \leq \sup_{z' \in \bigcup_{i=1}^k \overline{\phi_{A_i}(D)}} |f(z')|$$
where the right hand side is finite because $\bigcup_{i=1}^k \overline{\phi_{A_i}(D)}$ is a compact subset of $D$. Therefore, 
$$\sup_{z \in D}|f(z)| \leq\sup_{z' \in \bigcup_{i=1}^k \overline{\phi_{A_i}(D)}} |f(z')|=|f(z_0)|$$
for some $z_0 \in \bigcup_{i=1}^k\overline{\phi_{A_i}(D)} $. By the maximum-modulus principle, $f$ is constant on $D$. By the same argument we see that there can be no other eigenvalues of modulus 1.

Therefore it remains to show that $1$ is an algebraically simple eigenvalue. We need to show that $\ker(\l_0-\id)^2$ is one dimensional (so only consists of the constant functions). For a contradiction suppose that there exists $f \in H^2(D)$ for which $(\l_0-\id)f \neq 0$ but $(\l_0-\id)f \in \ker(\l_0-\id)$. So in particular $(\l_0-\id)f = c \one$ for some constant $c$. In particular, $c \neq 0$ since $(\l_0-\id)f \neq 0$ and therefore by replacing $f$ by $c^{-1} f$ we obtain that $(\l_0-\id)f=\one$, that is, $\l_0f= \one +f$. By induction we see that 
\begin{eqnarray}\l_0^n f= n \one+f. \label{ind} \end{eqnarray}

On the other hand, define 
$$\Gamma^n= \overline{ \bigcup_{A \in \A^n} \phi_{A}(D)}$$
and define
$\Gamma= \bigcap_{n=1}^{\infty} \Gamma^n$. Since $\Gamma^n$ is a nested sequence of compact subsets of $D$, $\Gamma$ is a compact subset of $D$. For any $z \in \Gamma$, 
\begin{eqnarray}
|\l_0^nf(z)| &=& \left| \sum_{A \in \A^n} p_{A}  f(\phi_{A}(z))\right| \\
&\leq & \sup_{z \in \Gamma} |f(z)|.
\end{eqnarray}
By (\ref{ind}), $|\l_0^n f(z)|=|n+f(z)| \geq n-|f(z)|$ implying that
$$n \leq 2\sup_{z \in \Gamma} |f(z)|$$
which is clearly a contradiction since $f$ is bounded on $\Gamma$.

\end{proof}

For each $i \in \I$ let $w_{A_i}: D \to \C$ denote the complex extension of $\tilde{w}_{A_i}$ to $D$ given by
$$w_{A_i}(z)= (a_i+c_i-b_i-d_i)z+b_i+d_i.$$
For any $A=A_{i_1} \cdots A_{i_n} \in \A^n$ define
$$w_A=w_{A_{i_1}}(\phi_{A_{i_2} \cdots A_{i_n}})w_{A_{i_2}}(\phi_{A_{i_3} \ldots A_{i_n}}) \cdots w_{A_{i_n}}.$$ By (\ref{identity}), for all $x \in (0,1)$ and $A \in \A^n$
\begin{eqnarray}
w_{A}(x)= \langle A \x, \u \rangle  \label{wi dot}
\end{eqnarray}
where $\x=(x,1-x)$, $\u=(1,1)$.

Since $w_{A_i}$ maps $D$ to the disc centred in the real axis whose boundary passes through the points $a_i+c_i$ and $b_i+d_i$ it follows that for all $z \in D$,
\begin{eqnarray}\min\{a_i+c_i, b_i+d_i\} \leq |w_{A_i}(z)| \leq \max\{a_i+c_i, b_i+d_i\}.
\label{bounds}
\end{eqnarray}

For $|t|>0$, notice that $w_{A}(z)^t= \exp(t\log w_A(z))$ defines a holomorphic function from $D$ to $\C$, where $\log$ is understood as the unique holomorphic function from the right half plane to $\C$ such that $\exp \log z =z$ and $\log 1=0$. Since $\Re(w_A(z))>0$ for all $z \in D$ this extension is well-defined. 

For $|t|>0$ we define the perturbed transfer operator $\l_t: H^2(D) \to H^2(D)$ by
$$\l_tf=\sum_{i=1}^k p_iw_{A_i}(z)^t f \circ \phi_{A_i}.$$

Note that since $\phi_{A_i}$ are holomorphic self maps of $D$ and $w_{A_i}^t$ are bounded holomorphic functions on $D$ it again follows that $\l_t f \in H^2(D)$, see \cite[page 11]{shapiro}.

Also notice that
$$\l_t^n f= \sum_{A \in \A^n } p_{A}  w_{A}^t f \circ \phi_{A}.$$

Directly from Lemma \ref{dot norm} and (\ref{wi dot}) we see that the exponential growth rate of $w_{A_{i_1} \cdots A_{i_n}}(x)$ at a point $x \in (0,1)$ will be the same as the exponential growth rate of the norm $\norm{A_{i_1} \cdots A_{i_n}}$ since for any $A \in \A^n$
\begin{eqnarray}
c^{-1}\norm{A} \leq w_{A}(x) \leq c\norm{A}
\label{wi norm}
\end{eqnarray}
for a uniform constant $c$ which is independent of $n$ and $A$. This is precisely the property that will allow us to relate $\Lambda$ to the spectral properties of $\l_0$. 

Notice that if the constant $C_1$ defined in Assumption \ref{ass} was equal to 1, then for arbitrary $x \in (0,1)$ and all $A \in \A^n$, $w_{A}(x)=1$, which implies that $\Lambda=0$ by (\ref{wi norm}).

\section{Approximations of $\Lambda$}\label{sketch}

The following proposition establishes the link between the Lyapunov exponent and the spectral properties of $\l_0$.

\begin{prop}
Let $\lambda_1(t)$ denote the top eigenvalue of $\l_t$. There exists an open neighbourhood $U \subset \C$ of 0 such that $\lambda_1(t)$ is holomorphic for $t \in U$. Moreover
\begin{eqnarray}
\Lambda= \lambda_1'(0). \label{premain}
\end{eqnarray}
\label{pre main prop}
\end{prop}

\begin{proof}
It will follow from lemma \ref{l0 sing} that $\l_0$ is trace-class and therefore compact. By proposition \ref{simple}, $\lambda_1(0)$ is a simple eigenvalue. Since $t \mapsto \langle \l_tf,g\rangle_{H^2}$ is analytic for all $f, g \in H^2(D)$, it follows that $t \mapsto \l_t$ is analytic in $t$. Therefore, standard analytic perturbation theory arguments can be used to prove the first part.

By applying Theorem \ref{apt} with $L_t=\l_t$ immediately implies that $\lambda_1(t)$ is holomorphic in $t$. 

Next put
$$\mathcal{P}_t= \frac{1}{2\pi i} \int_{\gamma} (s \id -\l_t)^{-1} \textup{d}s $$
as in (c) of Theorem \ref{apt}. (a)-(c) of Theorem \ref{apt} imply that the image of $\mathcal{P}_t$ is an eigenspace for the eigenvalue $\lambda_1(t)$ and that $h_t= \mathcal{P}_t \one$ is an eigenfunction for the eigenvalue $\lambda_1(t)$. Note that $h_0= \one$. Since $t \mapsto \mathcal{P}_t$ is holomorphic it immediately follows that $t \mapsto h_t$ is also holomorphic. We write $g_0= \frac{\textup{d}}{\textup{d}t} h_t \bigr|_{t=0} \in H^2(D)$.

Fix some $z_0 \in \Gamma \cap (0,1)$. To deduce  (\ref{premain}), observe that for each $n>1$ and $t \in U$,
\begin{eqnarray*}
\lambda_1(t)^n h_t(z_0)&=& (\l_t^n h_t)(z_0)= \sum_{A \in \A^n} p_{A} w_{A}(z_0)^t h_t(\phi_{A}(z_0)) \\
&=& \sum_{A \in \A^n} p_{A} \exp(t \log w_{A}(z_0)) h_t(\phi_{A}(z_0)) .
\end{eqnarray*}
Differentiating at $t=0$ we obtain
\begin{eqnarray*}
n\lambda_1'(0) + g_0(z_0)&=& \sum_{A \in \A} p_{A} \log w_{A}(z_0)+ \sum_{A \in \A^n} p_{A} g_0(\phi_{A}(z_0)) \\
&=& \sum_{A \in \A^n} p_{A} \log w_{A}(z_0) +\l_0^n g_0(z_0)
\end{eqnarray*}
where we used that $\lambda_1(0)=1$ and $h_0= \one$. Therefore
$$\left|n \lambda_1'(0)- \sum_{A \in \A^n} p_{A} \log w_{A}(z_0)\right| = |\l_0^n g_0(z_0)-g_0(z_0)| \leq 2 \sup_{z_0 \in \Gamma}|g_0(z_0)|$$
which is finite due to the compactness of $\Gamma$.
By (\ref{wi norm}), there exists some uniform constant $C$ such that for all $A \in \A^n$,
$$\log \norm{A} -C \leq \log w_{A}(z_0) \leq \log \norm{A} +C $$
and therefore
$$\Lambda= \lim_{n \to \infty} \frac{1}{n} \sum_{A \in \A^n} p_{A} \log w_{A}(z_0)= \lambda_1'(0).$$
\end{proof}

In section \ref{est} we will show that for all $t \in \C$ the approximation numbers $s_n(\l_t)$ decay exponentially and therefore $\l_t$ is trace-class, meaning that the determinant $\det(\id-z\l_t)$ is defined and is an entire function of $z$ which is given in the form
$$\det(\id -z\l_t)= \sum_{n=0}^{\infty} b_n(t)z^n$$
for $b_n(t) \in \C$. Therefore, denoting $\lambda_n(t)$ to be the $n$th eigenvalue of the operator $\l_t$ and observing that the zeroes of the determinant $\det(\id-z\l_t)$ are the reciprocals of the eigenvalues of $\l_t$ it follows that
\begin{eqnarray}
\sum_{n=0}^{\infty} b_n(t)\lambda_1(t)^{-n}=0. \label{key} \end{eqnarray}
Therefore, provided the coefficients $b_n(t)$ are holomorphic with respect to $t$, we can differentiate (\ref{key}) with respect to $t$ and obtain $\Lambda=\lambda_1'(0)$ in terms of $b_n(0)$ and $b_n'(0)$. The following lemma provides us with an expression for the coefficients $b_n(t)$.

\begin{prop}
For all $t \in \C$, $\l_t$ is trace-class. In particular $\det(\id-z\l_t)$ is an entire function of $z$ and is given in the form
$$\det(\id -z\l_t)= \sum_{n=0}^{\infty} b_n(t)z^n$$
for $b_n(t) \in \C$ where $b_0(t)=1$ for all $t$ and for $n\geq 1$ is defined as
\begin{eqnarray}b_n(t)= \sum_{m=1}^n \frac{(-1)^m}{m!} \sum_{\substack{n_1, \ldots, n_m \in \N^m \\ n_1 + \ldots +n_m=n}} \prod_{i=1}^m \frac{\textnormal{tr} \l_t^{n_i}}{n_i}. \label{trace coef2} \end{eqnarray}
\label{tc prop}
\end{prop}

\begin{proof}
The fact that $\l_t$ is trace-class for each $t$ will follow from Lemma \ref{lt sing}. (\ref{trace coef2}) is a well-known result but we include its proof for completeness. Notice that
\begin{eqnarray*}
\sum_{l=0}^{\infty} \frac{(-1)^l}{l!} \left( \sum_{n=1}^{\infty} \frac{z^n \tr \l_t^n}{n}\right)^l &=& \exp\left(-\sum_{n=1}^{\infty} \frac{ \tr(z\l_t)^n}{n} \right)\\
&=& \exp\left( -\sum_{n=1}^{\infty} \frac{1}{n} \sum_{k=1}^{\infty} (z\lambda_k(\l_t))^n \right) \\
&=& \exp\left( -\sum_{k=1}^{\infty} \sum_{n=1}^{\infty} \frac{(z\lambda_k(\l_t))^n}{n} \right) \\
&=& \prod_{k=1}^{\infty} \exp\left(- \sum_{n=1}^{\infty} \frac{(z\lambda_k(\l_t))^n}{n}\right)\\
&=& \prod_{k=1}^{\infty}\exp(\log(1-z\lambda_k(\l_t)))= \det(1-z\l_t)
\end{eqnarray*}
where the rearrangement on the third line is permitted because $\sum_{n=1}^{\infty} \frac{1}{n} \sum_{k=1}^{\infty} (z\lambda_k(\l_t))^n$ is absolutely summable since if $|z|< |\lambda_1(\l_t)|$ then
$$\sum_{n=1}^{\infty} \sum_{k=1}^{\infty}|(z\lambda_k(\l_t))^n| \leq \sum_{n=1}^{\infty} \sum_{k=1}^{\infty} |z|^n s_k((\l_t)^n)=\sum_{n=1}^{\infty} |z|^n |\l_t^n|_{\text{tr}} \leq \sum_{n=1}^{\infty} |z|^n |\l_t|_{\text{tr}} \norm{\l_t^{n-1}}$$
which converges by Gelfand's formula. The result follows by equating coefficients.
\end{proof}

We next we obtain a simpler expression for the trace of $\l_t^m$. In view of (\ref{trace coef2}) this expression clearly implies that the coefficients $b_n(t)$ are holomorphic in $t$. (The holomorphicity of $b_n$ can alternatively be understood as a consequence of the holomorphicity of the map $t \mapsto \l_t$ as a function taking values in the Banach space of trace-class operators equipped with the trace norm, but that perspective will not be required in our proof.) 

\begin{prop}
For every $t \in \C$ and $m \in \N$,
$$\tr(\l_t^m)= \sum_{A \in \A^m} p_{A} \lambda_1(A)^t \left(1-\frac{\lambda_2(A)}{\lambda_1(A)}\right)^{-1}.$$

\label{trace exp}
\end{prop}

\begin{proof}
For each $m \in \N$, $A \in \A^m$ and $ t \in \C$ let $\l_{A, t}f= w_{A}^t \cdot f \circ \phi_A$ so that $\l_t= \sum_{A \in \A} p_{A} \l_{A,t}.$ Let $z_{A}$ denote the unique fixed point of $\phi_{A}$. Let $\lambda_1(A)$ and $\lambda_2(A)$ denote the eigenvalues of the matrix $A$. We begin by showing that $\frac{\textup{d}}{\textup{d}z}\phi_{A}(z_{A})= \frac{\lambda_2(A)}{\lambda_1(A)}.$ Let $A= \begin{pmatrix} a&b \\ c&d \end{pmatrix}$ so that 
$$\frac{\textup{d}}{\textup{d}z} \phi_{A}(z_{A})= \frac{a-b-z_{A}(a+c-b-d)}{(a+c-b-d)z_{A}+b+d}.$$
Clearly the denominator is 
$$(a+c-b-d)z_{A}+b+d=w_{A}(z_{A})= \left(A \begin{pmatrix} z_{A} \\1-z_{A} \end{pmatrix}\right) \cdot  \begin{pmatrix} 1\\1 \end{pmatrix}= \lambda_1(A) \begin{pmatrix} z_{A} \\1-z_{A} \end{pmatrix}\cdot  \begin{pmatrix} 1\\1 \end{pmatrix}=\lambda_1(A).$$
The numerator is given by
\begin{eqnarray*}
a-b-z_{A}(a+c-b-d)&=&\begin{pmatrix} d&-b \\-c&a \end{pmatrix} \begin{pmatrix} z_{A} \\1-z_{A} \end{pmatrix} \cdot  \begin{pmatrix} 1\\1 \end{pmatrix} = \det(A) A^{-1} \begin{pmatrix} z_{A} \\1-z_{A} \end{pmatrix} \cdot  \begin{pmatrix} 1\\1 \end{pmatrix} \\
&=& \frac{\lambda_1(A) \lambda_2(A)}{\lambda_1(A)} A^{-1} A\begin{pmatrix} z_{A} \\1-z_{A} \end{pmatrix} \cdot  \begin{pmatrix} 1\\1 \end{pmatrix} =\lambda_2(A)
\end{eqnarray*}
 Therefore by \cite[Theorem 4.2]{bj} \footnote{ The trace formula (\ref{tf}) was first obtained by Ruelle \cite{ruelle} for operators that acted on the Banach space of holomorphic functions on $D$ that extend continuously to the closure of $D$. In \cite{bj} Bandtlow and Jenkinson showed that the same formula holds for operators on more general spaces of analytic functions, including Bergman spaces and the Hardy-Hilbert spaces.} the trace of $\l_{A, t}$ is given by
\begin{eqnarray}
\tr(\l_{A,t})=\frac{\lambda_1(A)^t}{1-\frac{\lambda_2(A)}{\lambda_1(A)}}. \label{tf}
\end{eqnarray}
Since $\tr\left(\sum_{A \in \mathcal{A}^m} \l_{A,t}\right)=\sum_{A \in \mathcal{A}^m} \tr(\l_{A,t})$, the result follows.
\end{proof}

By (\ref{trace coef2}) and Proposition \ref{trace exp} it is clear to see that for each $t \in \C$ and $n \in \N$ the coefficients $b_n(t)$ can be computed explicitly and require one to calculate all possible random products of $m$ matrices from $\{A_1, \ldots, A_k\}$ for each $1 \leq m \leq n$. It also immediately follows from (\ref{trace coef2}) and Proposition \ref{trace exp} that for each $n \in \N$, $b_n(t)$ is holomorphic in $t$. By combining these facts with Proposition \ref{pre main prop}, we can establish the link between the Lyapunov exponent and the determinant of $\l_t$.

\begin{prop}
Let $\det(\id -z\l_t)= \sum_{n=0}^{\infty} b_n(t)z^n$ as before. Then
\begin{eqnarray}
\Lambda= \lambda_1'(0)= \frac{\sum_{n=0}^{\infty} b_n'(0)}{\sum_{n=0}^{\infty} nb_n(0)}. \label{main2}
\end{eqnarray} \label{main prop} \end{prop}

\begin{proof}
The first equality follows from Proposition \ref{pre main prop}. To deduce the second equality in (\ref{main2}) observe that since the zeroes of the determinant $\det(\id-z\l_t)$ are the reciprocals of the eigenvalues of $\l_t$,
\begin{eqnarray}\sum_{n=0}^{\infty} b_n(t)\lambda_1(t)^{-n}=0.\label{root} \end{eqnarray} 
It will follow from (\ref{bnt}) that $|b_n(t)|=O(\exp(-cn^2))$ uniformly on $U$ and therefore by applying the Cauchy integral formula we deduce that the partial sums $\sum_{n=1}^N b_n'(t)-nb_n(t)\lambda_1'(t)$ converge uniformly on compact subsets of $U$ as $N \to \infty$. Therefore we can differentiate (\ref{root}) and take derivatives inside the summation to obtain 
\begin{eqnarray}
0= \frac{\textup{d}}{\textup{d}t} \left(\sum_{n=0}^{\infty} b_n(t) \lambda_1(t)^{-n} \right) \biggr|_{t=0}= \sum_{n=0}^{\infty} b_n^{\prime}(0)-nb_n(0)\lambda_1'(0). \label{ting} \end{eqnarray}
Since $(\lambda_1(0))^{-1}$ is a simple zero of $\det(\id-z \l_0)$, it follows that $\sum_{n=0}^{\infty} nb_n(0) \neq 0$ and so by rearranging (\ref{ting}) we obtain
$$\lambda_1'(0)= \frac{\sum_{n=0}^{\infty} b_n'(0)}{\sum_{n=0}^{\infty} nb_n(0)}$$
which completes the proof.
\end{proof}

Since as it was noted earlier $b_n(t)$ can be computed for small $n$ (meaning that by the Cauchy integral formula $b_n'(0)$ can also be computed for small $n$), (\ref{main2}) provides us with natural candidates for approximating $\Lambda$ given by
\begin{equation}
\Lambda_N=\frac{\sum_{n=0}^{N} b_n'(0)}{\sum_{n=0}^{N} nb_n(0)}. \label{approx}
\end{equation}
Observe that the approximation $\Lambda_N$ of $\Lambda$ corresponds to truncating the determinant $\det(\id-z\l_t)$ after $N+1$ terms, yielding
$$\sum_{n=0}^N b_n(t) \lambda_1(t)^{-n} \approx 0,$$
followed by differentiating at $t=0$ and solving for $\lambda_1'(0)$.

Now, using Proposition \ref{trace exp} we can  define
$$t_m:= \tr(\l_0^m)=\sum_{A \in \A^m} p_{A}\left(1-\frac{\lambda_2(A)}{\lambda_1(A)}\right)^{-1}$$
and
$$\tau_m:= \frac{\textup{d}}{\textup{d}t} \tr(\l_t^m) \Bigr|_{t=0}= \sum_{A \in \A^m} p_{A}\log \lambda_1(A)\left(1-\frac{\lambda_2(A)}{\lambda_1(A)}\right)^{-1}.$$
Then by (\ref{trace coef2}) and (\ref{approx}) we see that $b_n'(0)=\alpha_n$ and $b_n(0)=a_n$ and therefore $\Lambda_N$ is given as in Theorem \ref{main}.

\subsection{Effective estimates}
 
It remains for us to obtain explicit bounds on the error $|\Lambda_N- \Lambda|$. Evaluating this difference we see that
\begin{eqnarray}
|\Lambda_N - \Lambda| = \left| \frac{ \sum_{n=0}^N b_n'(0) \sum_{n=0}^{\infty} nb_n(0) - \sum_{n=0}^{\infty} b_n'(0) \sum_{n=0}^N nb_n(0) }{\sum_{n=0}^N nb_n(0) \sum_{n=0}^{\infty} nb_n(0)} \right|. \label{diff}
\end{eqnarray}
The denominator can be written as 
\begin{eqnarray}\sum_{n=0}^N nb_n(0) \sum_{n=0}^{\infty} nb_n(0) =\left(\sum_{n=0}^{\infty} nb_n(0)- \sum_{n=N+1}^{\infty} nb_n(0)\right) \sum_{n=0}^{\infty} nb_n(0) \label{denom} \end{eqnarray}
therefore we need an upper bound on $|b_n(0)|$ and a lower bound on $\left|\sum_{n=1}^{\infty} nb_n(0)\right|$. The numerator can be written as
\begin{align}
\sum_{n=0}^N b_n'(0) \sum_{n=0}^{\infty} nb_n(0) - \sum_{n=0}^{\infty} b_n'(0) \sum_{n=0}^N nb_n(0) &=&  \left( \sum_{n=0}^{\infty} b'_n(0)-\sum_{n=N+1}^{\infty} b'_n(0)\right) \sum_{n=0}^{\infty} nb_n(0) \nonumber \\ & & - \sum_{n=0}^{\infty} b'_n(0)\left( \sum_{n=0}^{\infty} nb_n(0)-\sum_{n=N+1}^{\infty} nb_n(0)\right) \label{num}
\end{align}
and therefore we also need upper bounds on $|b_n'(0)|$.

In summary, we are looking for effective \emph{upper bounds} on $|b_n(0)|$ and $|b_n'(0)|$ and an effective \emph{lower bound} on $|\sum_{n=1}^{\infty}n b_n(0)|$.

In order to estimate $b_n(0)$ recall that by (\ref{sing coef}), 
\begin{eqnarray}
|b_n(t)| \leq \sum_{i_1< \ldots< i_n} |\lambda_{i_1}(\l_t) \ldots \lambda_{i_n}(\l_t)| \leq \sum_{i_1< \ldots < i_n} s_{i_1}(\l_t) \ldots s_{i_n}(\l_t).\label{sing bound}
\end{eqnarray}
Therefore we will obtain explicit upper bounds on $b_n(0)$ by estimating the approximation numbers $s_n(\l_0)$.

Let $l>0$. In order to estimate $|b_n'(0)|$ recall that by the Cauchy integral formula
$$b_n^{\prime}(0)= \frac{1}{2\pi i} \int_{S^{l}} \frac{b_n(t)}{t^2} \textup{d}t$$
where $S^{l}$ denotes the circle of radius $l>0$ centred at 0. Therefore, an upper bound on $|b_n'(0)|$ corresponds to an upper bound on $|b_n(t)|$ for $|t|=l$, which in view of (\ref{sing bound}) boils down to estimating the approximation numbers $s_n(\l_t)$ for $|t|=l$. Since $b_n(t)$ is holomorphic for all $t$, in principle $l$ can be chosen to be any real number. However in practice we will choose it in such a way that we minimise the upper bound on $|b_n'(0)|$.

Finally, for the lower estimate on $|\sum_{n=1}^{\infty} nb_n(0)|$ observe that
\begin{eqnarray}\left|\sum_{n=1}^{\infty} nb_n(0)\right|= \left|\frac{\textup{d}}{\textup{d}z} \sum_{n=0}^{\infty} b_n(0)z^n \biggr|_{z=1} \right| = \left|\frac{\textup{d}}{\textup{d}z} \prod_{n=1}^{\infty}(1-z\lambda_n(0)) \biggr|_{z=1}\right|=\prod_{n=2}^{\infty}|1-\lambda_n(0)|
\label{sum bound}
\end{eqnarray}
where derivatives can be taken outside of the summation due to uniform convergence of the partial sums $\sum_{n=1}^N nb_n(0)z^{n-1}$ for $|z|<\epsilon<1$ and the final equality follows by the chain rule and the fact that $\lambda_1(0)=1$. For sufficiently large $n$ we'll apply (\ref{product ineq}) to deduce that 
$$\prod_{i=1}^n |\lambda_n(\l_0)| \leq \prod_{i=1}^n s_n(\l_0)$$
which will allow us to use our approximation number estimates to obtain a lower bound for all sufficiently large terms in the product on the right hand side of (\ref{sum bound}). For small $n$, we'll bound $|1-\lambda_n(0)| \geq 1-|\lambda_n(0)|\geq 1-|\lambda_2(0)|$, which means we need to obtain an explicit upper bound (which is strictly less than 1) for the second eigenvalue of $\l_0$.

Therefore the efficiency of the algorithm essentially depends on the eigenvalues $\{\lambda_n(t)\}_{n \in \N}$ which in turn depend on the singular values $\{s_n(\l_t)\}_{n \in \N}$. We will see in the next section that both of these are decaying exponentially at the rate $O(r^n)$.

\section{Estimates} \label{est}

\subsection{Estimates on $|b_n(0)|$ and $|b_n'(0)|$} 

We begin by estimating the approximation numbers of the operator $\l_0$.

\begin{lma}
Let $r$ be given by (\ref{r}). For every $n \in \N$,
$$s_{n+1}(\l_0) \leq \frac{1}{\sqrt{1-r^2}} r^n.$$
\label{l0 sing}
\end{lma}

\begin{proof}
Since $\phi_{A_i}$ is a linear fractional transformation, $\phi_{A_i}(D)=D_i$ where $D_i$ is the disk centred in the real axis whose boundary passes through $\frac{b_i}{b_i+d_i}$ and $\frac{a_i}{a_i+c_i}$. Let $f \in H^2(D)$ so that
$$f(z)= \sum_{n=0}^{\infty} 2^n\alpha_n(f) \left(z-\frac{1}{2}\right)^n.$$
Then 
$$f\circ \phi_{A_i}(z)= \sum_{n=0}^{\infty} 2^n\alpha_n(f)\left(\phi_{A_i}(z)-\frac{1}{2}\right)^n.$$
Notice that $|\phi_{A_i}(z)-\frac{1}{2}| \leq \frac{R_i}{R_i+1}-\frac{1}{2}$ where $R_i$ was defined in (\ref{R}). Put $u_{0,i}(z)=1$,
$$u_{1,i}(z)=2 (\phi_{A_i}(z)-\frac{1}{2})$$
and $u_{n, i}(z)=(u_{1,i}(z))^n$. It is easy to see that $u_{1, i}(D) \subseteq B(0,r)$. It follows that
$$f \circ \phi_{A_i}(z)= \sum_{n=0}^{\infty} \alpha_n(f) u_{n,i}(z).$$
Now, put 
$$\l_0^{(N)} f=\sum_{n=0}^{N-1} \sum_{i=1}^k p_i \alpha_n(f) u_{n,i}(z).$$
$\l_0^{(N)}$ clearly has rank not greater than $N$ and
\begin{eqnarray*}
\norm{\l_0 f-\l_0^{(N)}f}_{H^2} &\leq& \norm{\l_0 f-\l_0^{(N)}f}_{\infty} \\
&\leq& \sum_{n=N}^{\infty} \sum_{i=1}^k p_i |\alpha_n(f)| \norm{u_{n,i}}_{\infty} \\
&\leq& \left(\sum_{n=N}^{\infty} \alpha_n(f)^2 \right)^{\frac{1}{2}} \left( \sum_{n=N}^{\infty} \left(\sum_{i=1}^k p_i\norm{u_{n, i}}_{\infty}\right)^2\right)^{\frac{1}{2}} \\
&\leq&  \norm{f}_{H^2} \left(\sum_{n=N}^{\infty} r^{2n} \right)^{\frac{1}{2}} \\
&\leq&  \frac{r^N}{\sqrt{1-r^2}}\norm{f}_{H^2}
\end{eqnarray*}
where the third inequality follows by H\"older's inequality. Therefore
$$s_{n+1}(\l_0) \leq  \frac{1}{\sqrt{1-r^2}}r^n  $$
completing the proof.

\end{proof}

It is now easy to estimate the approximation numbers of $\l_t$ for any $|t| \geq 0$.

\begin{lma}
Let $r$ be given by (\ref{r}), $C_1$ be given by (\ref{c1}) and $\theta$ be given by (\ref{theta}). Denote
\begin{eqnarray}
C_2=\sqrt{(\log C_1)^2+\theta^2}. \label{c2}
\end{eqnarray}
Then for any $|t|\geq 0$,
$$s_{n+1}(\l_t) \leq \frac{1}{\sqrt{1-r^2}} r^ne^{C_2|t|}.$$
\label{lt sing}
\end{lma}

\begin{proof}
Put
$$\l_t^{(N)} f= \sum_{n=0}^{N-1} \sum_{j=1}^k   p_j ((a_j+c_j-b_j-d_j)z+b_j+d_j)^t \alpha_n(f) u_{n,j}(z)$$
which is an operator of rank at most $N$. Similarly to before,
\begin{eqnarray*}
\norm{\l_t f-\l_t^{(N)}f}_{H^2} &\leq&  \frac{r^N}{\sqrt{1-r^2}}\norm{f}_{H^2} \sup_{1 \leq j \leq k} \sup_{z \in D}  \left|((a_j+c_j-b_j-d_j)z+b_j+d_j)^t\right|. \\
\end{eqnarray*}
Note that $\left|((a_j+c_j-b_j-d_j)z+b_j+d_j)^t\right|=|w_{A_j}(z)^t|$. Let $\Im(z)$ denote the imaginary part of $z$ and $\Arg(z)$ denote the argument of $z$. Then
\begin{eqnarray*}
w_{A_j}(z)^t &=& \exp((\Re(t)+i\Im(t))(i\Arg(w_{A_j}(z))+\log|w_{A_j}(z)|)) \\
&=& |w_{A_j}(z)|^{\Re(t)} \exp(-\Im(t)\Arg(w_{A_j}(z)))|w_{A_j}(z)|^{i\Im(t)} \exp(i \Re(t)\Arg(w_{A_j}(z))).
\end{eqnarray*}
At this point we could use the fact that $\Arg(w_{A_j}(z)) \in (-\frac{\pi}{2}, \frac{\pi}{2})$ to obtain the bound
$$\sup_{1 \leq j \leq k}\sup_{z \in D} |w_{A_j}(z)^t| \leq \exp\left(|t|\sqrt{(\log C_1)^2+\frac{\pi^2}{4}}\right)$$
 by using (\ref{bounds}) and the Cauchy-Schwarz inequality. Instead we choose to optimise this bound by obtaining an improved upper estimate on $\theta_j=\sup_{z \in D}| \Arg(w_{A_j}(z))|$. Since $w_{A_j}(D)$ is a disk which is centred in the real line whose boundary $S$ passes through the points $a_j+c_j$ and $b_j+d_j$, $\theta_j$ will correspond to the angle between the real axis and the unique tangent to the circle $S$ which passes through the origin and has a positive gradient. In particular, since the midpoint of $w_{A_j}(D)$ is $\frac{a_j+b_j+c_j+d_j}{2}$ and $w_{A_j}(D)$ has radius $\frac{|a_j+c_j-b_j+d_j|}{2}$ it follows that 
$$\sin \theta_j=\frac{|a_j+c_j-b_j+d_j|}{a_j+b_j+c_j+d_j}$$
and therefore for all $z \in D$ and $1 \leq j \leq k$,
$$\Arg w_{A_j}(z) \in (-\theta, \theta).$$
In particular by  (\ref{bounds}) and the Cauchy-Schwarz inequality,
$$|w_{A_j}(z)^t| \leq \exp(|\Re(t)| \log |w_{A_j}(t)|+|\Im(t)|\theta) \leq \exp\left(|t|\sqrt{(\log C_1)^2+\theta^2}\right)$$
and therefore
$$\norm{\l_t f-\l_t^{(N)}f}_{H^2} \leq \frac{r^N}{\sqrt{1-r^2}}\norm{f}_{H^2}  e^{C_2|t|}.$$

\end{proof}

Using Lemma \ref{l0 sing} we can obtain an upper bound on $|b_n(0)|$. 

\begin{lma}
Let 
$$C_{0}=\frac{1}{r\sqrt{1-r^2}}.$$
Then for all $n \in \N$,
$$|b_n(0)| \leq \frac{C_{0}^nr^{\frac{n(n+1)}{2}}}{\prod_{i=1}^n (1-r^i)}.$$
\label{taylor bound}
\end{lma}

\begin{proof}
By (\ref{sing bound}) and Lemma \ref{l0 sing}
\begin{eqnarray*}
|b_n(0)| &\leq& \sum_{i_1< \ldots < i_n} s_{i_1}(\l_0) \ldots s_{i_n}(\l_0) \\
&\leq& C_{0}^n \sum_{i_1< \ldots <i_n} r^{i_1+ \ldots + i_n}.
\end{eqnarray*}
Therefore the result follows by repeated geometric summation.
\end{proof}

Using Lemma \ref{lt sing} we can obtain an upper bound on $|b_n^{\prime}(0)|$. 

\begin{lma}
For each $n \in \N$
$$|b_n^{\prime}(0)| \leq  \frac{neC_{0}^nr^{\frac{n(n+1)}{2}}C_2}{\prod_{i=1}^n (1-r^i)}.$$
\end{lma}

\begin{proof}
By using Lemma \ref{lt sing} we can apply similar arguments to Lemma \ref{taylor bound} to deduce that for any $|t|>0$,
\begin{eqnarray}|b_n(t)| \leq \frac{C_{0}^ne^{C_2n|t|}r^{\frac{n(n+1)}{2}}}{\prod_{i=1}^n (1-r^i)}. \label{bnt}
\end{eqnarray}
Let $l>0$. Since $b_n(t)$ is holomorphic in $t$ everywhere, by the Cauchy integral formula,
$$b_n^{\prime}(0)= \frac{1}{2\pi i} \int_{S^{l}} \frac{b_n(t)}{t^2} \textup{d}t$$
where $S^{l}$ denotes the circle of radius $l>0$ centred at 0. Therefore,
$$|b_n'(0)| \leq \frac{1}{2\pi} \sup_{|t|=l} \left|\frac{b_n(t)}{l^2} \right| \cdot 2\pi l= \frac{\sup_{|t|=l} |b_n(t)|}{l}.$$

Since $e^{C_2}>1$,  $\frac{e^{C_2nl}}{l}$ has a unique minimum for $l \in (0,1)$. By differentiating $\frac{e^{C_2nl}}{l}$ with respect to $l$ and equating to 0 we see that the minimum of this expression is achieved at $l=\frac{1}{nC_2}$. Therefore
$$\sup_{|t|=\frac{1}{nC_2}} |b_n(t)| \leq \frac{eC_{0}^nr^{\frac{n(n+1)}{2}}}{\prod_{i=1}^n (1-r^i)}$$
from which the result follows.
\end{proof}

\subsection{Estimates on $|\sum_{n=1}^{\infty} nb_n(0)|$} \label{peres section}

We begin with an estimate on $|\lambda_2(0)|$. In \cite{peres}, Peres studied the operator 
$$\l f(\overline{\x}) = \sum_{i=1}^k p_if(A_i \cdot \overline{\x})$$
on the Banach space $B$ of $h$-Lipschitz functions $f: \R\P^1_{+} \to \R$. In the following lemma we will use ideas from \cite{peres} (in particular the proof of Theorem 1) to show that the absolute value of any eigenvalue $\lambda \neq 1$ of $\l:B \to B$ is bounded above by $\tau(\A,\p)$.

\begin{lma}Let $\tau(\A,\p)$ be  the weighted Birkhoff coefficient defined in (\ref{wbirkhoff}). Then any eigenvalue $\lambda$ of $\l$ satisfies $|\lambda| \leq \tau(\A,\p)$.
\label{peres} \end{lma}

\begin{proof}
Let $f \in B$ and $|f|_{B}$ denote the Lipschitz constant of $f$ (with respect to the Hilbert metric), that is, the minimum constant for which
$$|f(\overline{\x})- f(\overline{\y})| \leq |f|_B h(\overline{\x},\overline{\y}).$$
Since all of the matrices $A_i$ map the positive cone strictly inside itself, there exists $\eta< \infty$ such that 
$$\sup_{i,j \in \I}\sup_{ \overline{\x},\overline{\y} \in \R\P^1_{+}} h(\overline{A_i\x},\overline{A_{j}\y}) \leq \eta.$$
Let $\tau(A)$ denote the Birkhoff coefficient of the positive matrix $A$ as defined in (\ref{birkhoff}). Then for $\overline{\x},\overline{\y} \in \R\P^1_{+}$,  
\begin{eqnarray}
|(\l^nf)(\overline{\x})-(\l^nf)(\overline{\y})| &\leq& \sum_{i_1, \ldots, i_n \in \I} p_{i_1} \cdots p_{i_n} |f(\overline{A_{i_1 \ldots i_n}\x})-f(\overline{A_{i_1 \ldots i_n}\y})| \nonumber\\
&\leq& |f|_B\sum_{i_1, \ldots, i_n \in \I} p_{i_1} \cdots p_{i_n} \tau(A_{i_1}) \cdots \tau(A_{i_{n-1}})h(\overline{A_{i_n}\x},\overline{A_{i_n}\y}) \nonumber\\
&\leq& |f|_B \tau(\A,\p)^{n-1}\eta. \label{p1}
\end{eqnarray}
Next, observe that
\begin{eqnarray}
\l^{n+1}f(\overline{\x})-\l^nf(\overline{\x})= \sum_{j \in \I} p_j(\l^nf(\overline{A_j\x})-\l^nf(\overline{\x})) . \label{p2}
\end{eqnarray}
(\ref{p1}) and (\ref{p2}) imply that
\begin{eqnarray}
|\l^{n+1}f(\overline{\x})-\l^nf(\overline{\x})| \leq |f|_B \tau(\A,\p)^{n-1}\eta.\label{p3} \end{eqnarray}
Therefore $\l^nf(\overline{\x})$ is a Cauchy sequence and is convergent. Moreoever, since $|\tau(\A,\p)|<1$, by (\ref{p1}) the limits $\lim_{n \to \infty}\l^nf(\overline{\x})=\lim_{n \to \infty}\l^nf(\overline{\y})$ coincide for any $\overline{\x},\overline{\y} \in \R\P^1_{+}$ and therefore $\l^nf(\overline{\x})$ converges to a constant $c_f$ for all $\overline{\x} \in \R\P^1_{+}$. By (\ref{p3}),
\begin{eqnarray}
\norm{\l^{n+1}f-c_f}_{\infty} \leq |f|_B\eta \sum_{k=n}^{\infty} \tau(\A,\p)^k=C|f|_B \tau(\A,\p)^n \label{p4}
\end{eqnarray}
for some constant $C$ which is independent of $n$. Let $g \in B$ and $\lambda \neq 1$ such that $\l g=\lambda g$. Applying (\ref{p4}) to $g$ implies that
$$\norm{\lambda^{n+1}g-c_g}_{\infty} \leq C|g|_B\tau(\A,\p)^n$$
which means that $c_g=0$ and therefore
$$\norm{\lambda^{n+1}g}_{\infty} \leq C|g|_B\tau(\A,\p)^n.$$
It follows that $|\lambda| \leq \tau(\A,\p).$
\end{proof}

We now show that this implies that $\tau(\A,\p)$ is an upper bound on $|\lambda_2(0)|$.

\begin{lma}
 Let $\tau(\A,\p)$ be the weighted Birkhoff coefficient defined in (\ref{wbirkhoff}). Then $|\lambda_2(0)| \leq \tau(\A,\p)$. \label{second}
\end{lma}

\begin{proof}
Let $f \in H^2(D)$ such that $\l_0f= \lambda f$ for some $|\lambda|<1$. In particular $f$ is not a constant function.  We begin by showing that $f$ is Lipschitz with respect to the standard Euclidean metric. Since
$$f(z)= \lambda^{-1} \sum_{i=1}^k p_i f \circ \phi_{A_i}(z)$$
we can differentiate to obtain
$$f'(z)= \lambda^{-1} \sum_{i=1}^k p_i f'(\phi_{A_i}(z)) \phi_{A_i}'(z).$$
Since 
$$\phi_{A_i}'(z)=\frac{a_i-b_i-z(a_i+c_i-b_i-d_i)}{(a_i+c_i-b_i-d_i)z+b_i+d_i},$$
it is easy to see by (\ref{bounds}) that $\phi_{A_i}'$ is bounded on $D$ for all $i \in \{1, \ldots, k\}$. Moreoever since $|f'(\phi_{A_i}(z))| \leq \sup_{y \in \overline{\phi_{A_i}(D)}}|f'(y)|$ it follows that the derivative of $f$ is bounded on $D$, and so $f$ is Lipschitz. Let $\overline{f}: (0,1) \to \R$ denote the restriction of $f$ to the real line. Then $\overline{f}$ is also Lipschitz with respect to the Euclidean metric and we let $|\overline{f}|_{\text{Lip}}$ denote the corresponding Lipschitz constant of $\overline{f}$.

Let $g: \R \P^1_{+} \to \R$ be defined by $g(\overline{\x})=\overline{f}(x)$ where $x$ is defined uniquely as $x \in (0,1)$ for which $\overline{\x}=(x, 1-x)$. Since $f$ is not a constant function this implies $\overline{f} $ is not constant which in turn implies $g$ is not constant. Observe that $|x-y| < \log \left(\frac{x}{y}\right)<  \log \frac{x(1-y)}{y(1-x)} $ for $0<y<x<1$ and so by symmetry
\begin{eqnarray*}
|x-y| \leq  \left| \log \frac{x(1-y)}{y(1-x)} \right| 
\end{eqnarray*}
for any $x, y \in (0,1)$, with equality if and only if $x=y$. Therefore since $\overline{f}$ is Lipschitz with respect to the Euclidean metric, $g$ is Lipschitz with respect to the Hilbert metric $h$:
 $$|g(\overline{\x})-g(\overline{\y})|= |\overline{f}(x)-\overline{f}(y)| \leq |\overline{f}|_{\text{Lip}} |x-y| \leq  |\overline{f}|_{\text{Lip}} \left| \log \frac{x(1-y)}{y(1-x)}\right|= |\overline{f}|_{\text{Lip}} h(\overline{\x}, \overline{\y}).$$
  Therefore $g$ is an $h$-Lipschitz eigenfunction for $\l$ since
\begin{eqnarray*}
\l g(\overline{\x})&=& \sum_{i=1}^k p_i g(A_i \cdot \overline{\x}) \\
&=& \sum_{i=1}^k p_i g\left(\frac{(a_i-b_i)x+b_i}{(a_i+c_i-b_i-d_i)x+b_i+d_i}, \frac{(c_i-d_i)x+d_i}{(a_i+c_i-b_i-d_i)x+b_i+d_i}\right) \\
&=& \sum_{i=1}^k p_i \overline{f}\left(\frac{(a_i-b_i)x+b_i}{(a_i+c_i-b_i-d_i)x+b_i+d_i}\right) \\
&=&\l_0\overline{f}(x)=\lambda \overline{f}(x)= \lambda g(\overline{\x}).
\end{eqnarray*}
Since $g$ is not a constant it follows that $|\lambda| \leq |\lambda_2(\l)| \leq \tau(\A,\p)$ by Lemma \ref{peres}.
\end{proof}

We are now in a position to obtain a lower bound on $|\sum_{n=1}^{\infty} nb_n(0)|$.

\begin{lma}
There exists $M \in \N$ such that
$$\left|\sum_{n=0}^{\infty} nb_n(0)\right| \geq (1-|\lambda_2(0)|)^{M-2}\prod_{n=M}^{\infty} |1-C_0 r^{\frac{n+1}{2}}|>0.$$
\label{lb}
\end{lma}

\begin{proof}
Recall that
$$\sum_{n=1}^{\infty} nb_n(0)= \frac{\textup{d}}{\textup{d}z} \sum_{n=0}^{\infty} b_n(0)z^n \biggr|_{z=0} = \frac{\textup{d}}{\textup{d}z} \prod_{n=1}^{\infty}(1-z\lambda_k(0)) \biggr|_{z=0}.$$
By applying the chain rule and recalling that $\lambda_1(0)=1$ we see that
$$\left|\sum_{n=1}^{\infty} nb_n(0)\right|=\prod_{n=2}^{\infty}|1-\lambda_n(0)|.$$
Let $M \geq 2$ be sufficiently large that $C_0 r^{\frac{M+1}{2}}<1$. Then for each $n \geq M$
$$|\lambda_n(0)|^n \leq \prod_{k=1}^n |\lambda_k(0) |\leq \prod_{k=1}^n s_k(\l_0) \leq C_0^n \prod_{k=1}^n r^k=C_0^n r^{\frac{1}{2}n(n+1)}.$$
In particular, $|\lambda_n(0)| \leq C_0 r^{\frac{n+1}{2}} <1$ and so
$$\prod_{n=M}^{\infty}|1-\lambda_n(0)| \geq \prod_{n=M}^{\infty} |1-C_0 r^{\frac{n+1}{2}}|.$$
On the other hand, since $\lambda_n(0)$ are decreasing in absolute value,
$$\prod_{n=2}^{M-1} |1-\lambda_n(0)| \geq (1-|\lambda_2(0)|)^{M-2}.$$

\end{proof}

\subsection{Summary} \label{summary}

Recall that $r:= \max_{i \in \I} \frac{R_i-1}{R_i+1}$ where $R_i$ was the smallest constant for which $\frac{1}{R_i} \leq \frac{a_i}{c_i}, \frac{b_i}{d_i} \leq R_i$. Also denote $s:= \tau(\A,\p)$, the weighted Birkhoff contraction coefficient. Recall that $C_1:=\min_{i \in \I} \{a_i+c_i, b_i+d_i, \frac{1}{a_i+c_i}, \frac{1}{b_i+d_i}\}>1$ by Assumption \ref{ass}. Also recall that $\theta= \max_{1 \leq i \leq k} \{\arcsin(\frac{|a_i+c_i-b_i-d_i|}{a_i+b_i+c_i+d_i})\}$. Let $C_2= \sqrt{(\log C_1)^2+\theta^2}$, $C_0= \frac{1}{r\sqrt{1-r^2}}$ and $M\geq 2$ be large enough that $C_0r^{\frac{M+1}{2}}<1$.

Denote $\alpha= \sum_{n=0}^{\infty} nb_n(0)$, $\beta=\sum_{n=0}^{\infty} b_n'(0)$ and for each $N \in \N$ denote $\alpha_N= \sum_{n=N+1}^{\infty} nb_n(0)$ and $\beta_N=\sum_{n=N+1}^{\infty} b_n'(0)$. By (\ref{diff}), (\ref{denom}) and (\ref{num})
\begin{eqnarray*}
|\Lambda-\Lambda_N|&=& \left| \frac{(\beta-\beta_N) \alpha-\beta(\alpha-\alpha_N)}{(\alpha-\alpha_N)\alpha}\right| \\
&\leq& \frac{|\alpha||\beta_N|+ |\alpha_N||\beta|}{|\alpha||\alpha-\alpha_N|}
\end{eqnarray*}
where
\begin{eqnarray*}
\alpha^{-}=|1-s|^{M-2}\prod_{n=M}^{\infty}(1-C_0r^{\frac{n+1}{2}}) \leq |\alpha| &\leq& \sum_{n=1}^{\infty} n\frac{C_0^n r^{\frac{n(n+1)}{2}}}{\prod_{i=1}^n (1-r^i)} =\alpha^{+}\\
|\alpha_N| &\leq& \sum_{n=N+1}^{\infty} n\frac{C_0^n r^{\frac{n(n+1)}{2}}}{\prod_{i=1}^n (1-r^i)} =\alpha_N^{+}\\
|\beta| &\leq& \sum_{n=0}^{\infty} \frac{neC_2C_0^n r^{\frac{n(n+1)}{2}}}{\prod_{i=1}^n (1-r^i)}=\beta^{+} \\
|\beta_N| &\leq& \sum_{n=N+1}^{\infty} \frac{neC_2C_0^n r^{\frac{n(n+1)}{2}}}{\prod_{i=1}^n (1-r^i)}=\beta_N^{+}. 
\end{eqnarray*}
Therefore, as long as $N$ is sufficiently large that
$$\alpha_N^{+}=\sum_{n=N+1}^{\infty} n\frac{C_0^n r^{\frac{n(n+1)}{2}}}{\prod_{i=1}^n (1-r^i)} < |1-s|^{M-2}\prod_{n=M}^{\infty}(1-C_0r^{\frac{n+1}{2}})=\alpha^{-} $$
we can use the above estimates to bound 
\begin{eqnarray*}
|\Lambda-\Lambda_N| &\leq& \frac{\beta_N^{+}}{\alpha^{-}-\alpha_N^{+}} + \frac{\alpha_N^{+}\beta^{+}}{\alpha^{-}(\alpha^{-}-\alpha_N^{+})}.\end{eqnarray*}

In particular, note that the error satisfies the bound provided in Theorem \ref{main} for some constants $C$ and $\gamma$. Notice that due to the modification to the definition of $\alpha^-$, this bound in slightly better than the one given in the introduction, as long as $\tau(\A,\p) \neq r$ and $M>2$.

\section{Examples} \label{eg}

In this section we illustrate the algorithm in action with two examples which demonstrate how the performance of the algorithm and the upper bound on the error $|\Lambda-\Lambda_N|$ depend on the pair $(\A,\p)$. We also describe an approach which can be used to optimise the estimate on the error.

\begin{exam} \label{eg1}
Let $\p=(\frac{1}{2}, \frac{1}{2})$ and $\A=\{A_1,A_2\}=\left\{ \begin{pmatrix} 2&1 \\1&1 \end{pmatrix}, \begin{pmatrix} 3&1 \\2&1 \end{pmatrix}\right\}$. The matrices in $\A$ highly contract the positive quadrant (as demonstrated by the fact that $R_1=2$ and $R_2=\frac{3}{2}$ are both close to 1) which yields a small value for $r=\frac{1}{3}$. 

The table below demonstrates the output of the first ten iterates of the algorithm, given to 40 decimal places, and the corresponding upper bound on the error $|\Lambda-\Lambda_N|$. After 9 iterates of the algorithm the approximation appears to be accurate to around 39 decimal places. However, with our current bound on $|\Lambda-\Lambda_N|$, we can only rigorously justify 18 decimal places, and therefore we will perform a change of basis in order to improve on the bound (see below).

\vspace{2mm}
\begin{center}
\begin{tabular}{ |c|c|c| } 
 \hline
 $N$ & $\Lambda_N$ & Upper bound on $|\Lambda-\Lambda_N|$ \\ 
 \hline
 $1$ & $1.1323207013592984485818131912319549169181$ & $109.679$ \\ 
 $2$ & $1.1438057609617536317295772822737684626387$ & $253.078 $ \\ 
 $3$ & $1.1433094613369731162622336336724095207554$ & $9.05634 $ \\ 
 $4$ & $1.1433110357039283332222408377554188622939$ & $0.135028$ \\ 
 $5$ & $1.1433110351029192261291838354049305776777$ & $0.000704763 $ \\ 
 $6$ & $1.1433110351029492460476387448104729203942$ & $ 1.19547\cdot10^{-6}$ \\ 
 $7$ & $1.1433110351029492458432516598590310787720$ & $6.62462\cdot10^{-10} $ \\ 
 $8$ & $1.1433110351029492458432518536556142146672$ & $ 1.20473\cdot10^{-13}$ \\ 
 $9$ & $1.1433110351029492458432518536555882994021$ & $7.21309\cdot10^{-18} $ \\ 
 $10$ & $1.1433110351029492458432518536555882994025$ & $1.4252\cdot10^{-22} $ \\
 \hline

\end{tabular}
\end{center}

\end{exam}

In the above example, the estimate on $|\Lambda-\Lambda_N|$ was not optimal, essentially due to the fact that each matrix mapped the positive directions into a cone which was not symmetric in the direction given by the representative vector $(1,1)$. In such a situation it is possible to perform a change of basis and  optimise the value of $r$, hence optimising the estimate on the error $|\Lambda-\Lambda_N|$. 

Given $\A=\{A_i: 1 \leq i \leq k\}$ and $\lambda\neq 0$, let $\A_{\lambda}$ denote the set
$$\A_{\lambda}=\left\{ \begin{pmatrix} a_i& \lambda^2 b_i \\ \frac{c_i}{\lambda^2}& d_i \end{pmatrix} : 1 \leq i \leq k \right\},$$
noticing that 
\begin{eqnarray}
\begin{pmatrix} a_i& \lambda^2 b_i \\ \frac{c_i}{\lambda^2}& d_i \end{pmatrix}= \begin{pmatrix} \lambda & 0\\ 0& \frac{1}{\lambda} \end{pmatrix} \begin{pmatrix} a_i&  b_i \\ c_i& d_i \end{pmatrix} \begin{pmatrix} \frac{1}{\lambda} & 0  \\ 0& \lambda \end{pmatrix}. \label{conj}
\end{eqnarray}
Fix a probability vector $\p$. From (\ref{conj}) it is clear that $\Lambda(\A,\p)=\Lambda(\A_{\lambda}, \p)$ for any $\lambda \neq 0$; moreover the algorithm remains the same since the approximations $\Lambda_N$ only depend on the eigenvalues of finite matrix products, which are invariant under conjugation.  Therefore, the goal is to choose a value of $\lambda$ such that $r_{\lambda}=r_{\A_{\lambda}}$ is as small as possible. 

Let 
$$R_{\lambda}:= \max_{i \in \I}\left\{ \frac{a_i \lambda^2}{c_i}, \frac{c_i}{a_i \lambda^2}, \frac{\lambda^2b_i}{d_i}, \frac{d_i}{\lambda^2b_i}\right\}.$$
Fix $\lambda_0 \neq 0$ to be the value that minimises $R_{\lambda}$. If $\A_{\lambda_0}$ consists only of column stochastic matrices, we can conclude that $\Lambda=0$. Otherwise, we have a new value for the constant (\ref{c1}) given by
$$C_{1,\lambda_0}=\max_{1 \leq i \leq k} \left\{a_i+\frac{c_i}{\lambda_0^2}, \lambda_0^2b_i+d_i, \frac{\lambda_0^2}{\lambda_0^2a_i+c_i}, \frac{1}{\lambda_0^2b_i+d_i}\right\} >1$$
and we can obtain a new estimate on $|\Lambda-\Lambda_N|$ by entering the values $r_{\lambda_0}$ and $C_{1, \lambda_0}$ into the expressions in section \ref{summary}.

By applying this approach to example \ref{eg1}, the problem boils down to minimising the expression 
$$R_{\lambda}=\max\{2\lambda^2, \frac{1}{\lambda^2}\}$$
which is minimised when $\lambda_0 =2^{-\frac{1}{4}}$ producing $r_{\lambda_0}=\frac{\sqrt{2}-1}{\sqrt{2}+1}\approx 0.17$. We obtain $C_{1, \lambda_0} =3+2 \sqrt{2}$.  

Below we reproduce the table from example \ref{eg1}, this time with the improved upper bounds on the errors $|\Lambda_N-\Lambda|$. Observe that with this improved bound, we can now rigorously justify 32 decimal places of $\Lambda_9$.

\begin{center}
\begin{tabular}{ |c|c|c| } 
 \hline
 $N$ & $\Lambda_N$ & Improved upper bound on $|\Lambda-\Lambda_N|$ \\ 
 \hline
 $1$ & $1.1323207013592984485818131912319549169181$ & $ 305.614$ \\ 
 $2$ & $1.1438057609617536317295772822737684626387$ & $1.25027$ \\ 
 $3$ & $1.1433094613369731162622336336724095207554$ & $ 0.00810529$ \\ 
 $4$ & $1.1433110357039283332222408377554188622939$ & $ 8.902\cdot 10^{-6}$ \\ 
 $5$ & $1.1433110351029192261291838354049305776777$ & $ 1.61193\cdot 10^{-9} $ \\ 
 $6$ & $1.1433110351029492460476387448104729203942$ & $4.86927\cdot 10^{-14}$ \\ 
 $7$ & $1.1433110351029492458432516598590310787720$ & $2.47219\cdot 10^{-19} $ \\ 
 $8$ & $1.1433110351029492458432518536556142146672$ & $2.11988\cdot10^{-25} $ \\ 
 $9$ & $1.1433110351029492458432518536555882994021$ & $ 3.08032\cdot10^{-32}$ \\ 
 $10$ & $1.1433110351029492458432518536555882994025$ & $ 7.6026\cdot 10^{-40} $ \\
 \hline

\end{tabular}
\end{center}
\vspace{2mm}

\begin{exam}
Let $\p=(\frac{1}{2}, \frac{1}{2})$ and $\A=\{A_1,A_2\}=\left\{ \begin{pmatrix} 3&1 \\1&3 \end{pmatrix}, \begin{pmatrix} 5&2 \\2&5\end{pmatrix}\right\}$. This time $\A$ consists of matrices which only mildly contract the positive quadrant, which the higher value of $r=\frac{1}{2}$ reflects. Note also that $\tau(\A,\p)=\frac{13}{28}$. Since $\Lambda_N$ takes longer to converge, we also include approximations $\Lambda_{11}$ to $\Lambda_{15}$.

\vspace{2mm}
\begin{center}
\begin{tabular}{ |c|c|c| } 
 \hline
 $N$ & $\Lambda_N$ & Upper bound on $|\Lambda-\Lambda_N|$ \\ 
 \hline
 $1$ & $1.6474483954897545390942122098288722131112$ & $ 2746.58$ \\ 
 $2$ & $1.6029620162255035051574698919128471509232 $ & $2758.98 $ \\ 
 $3$ & $1.6559697370513636166814363906952389005330$ & $ 2850.74$ \\ 
 $4$ & $1.6671719385685105789810370493542955219116$ & $ 4904.89$ \\ 
 $5$ & $1.6660672588037147468953253378545484207346 $ & $295.52 $ \\ 
 $6$ & $1.6661027783284791857844224273852985042619$ & $ 5.60925$ \\ 
 $7$ & $1.6661022515053788423273670386553440044077 $ & $ 0.0576355$ \\ 
 $8$ & $ 1.6661022550990213376439090467415586661772$ & $ 0.000292276$ \\ 
 $9$ & $1.6661022550875848155166723158881986394236 $ & $7.3219 \cdot 10^{-7} $ \\ 
 $10$ & $1.6661022550876019742044196964403971038778 $ & $ 9.08085 \cdot 10^{-10}$ \\
 $11$& $1.6661022550876019619657305469157054370027$ & $5.58504 \cdot 10^{-13} $ \\
 $12$& $1.6661022550876019619699091800215128050474$ & $1.70563 \cdot 10^{-16} $ \\
 $13$& $1.6661022550876019619699084931251373353824$ & $ 2.58906 \cdot 10^{-20}$ \\
 $14$& $1.6661022550876019619699084931797685439844$ & $1.95502 \cdot 10^{-24} $ \\
 $15$& $1.6661022550876019619699084931797664328543$ & $7.34848 \cdot 10^{-29} $ \\
 \hline
\end{tabular}
\end{center}
\vspace{2mm}

The approximations $\Lambda_N$ appear to take a lot longer to converge than in example \ref{eg1}, for instance $\Lambda_{9}$ seems to be accurate only to about 12 decimal places (compared to around 39 decimal places in the previous example).

We also observe that in this example we experience a longer initial lag before we can rigorously justify the accuracy of our approximations; indeed it is not until the 7th step that the upper bound on the error drops below 1. However, the upper bound eventually `catches up' with the apparent convergence of the approximations, for example $\Lambda_{14}$ appears to be accurate to around 32 decimal places and we can rigorously justify around 24 of these decimal places. Note that in this example since each of the matrices map the positive directions into a cone which is already symmetric in the direction given by the representative vector $(1,1)$, the approach used earlier to optimise the bound on the error cannot be implemented.
 \label{eg2}
\end{exam}

\section{Higher dimensions}\label{further}

It is natural to ask whether effective estimates for Lyapunov exponents along the lines of Theorem \ref{main} may be obtained for positive matrices in higher dimensions. In this article we expressed the top Lyapunov exponent of a collection of $2\times 2$ matrices $A_1,\ldots,A_k$ and probabilty vector $p_1,\ldots,p_k$ by considering the matrices' projective action on the real interval
\[\left\{\begin{pmatrix}x\\1-x\end{pmatrix} \colon x \in (0,1)\right\}\]
and extending it to a holomorphic action on the set of complex vectors
\[\left\{\begin{pmatrix}z\\1-z\end{pmatrix} \colon \left|z-\frac{1}{2}\right|<\frac{1}{2}\right\}\]
which for simplicity we identified with the disc $D:=\{z \in \mathbb{C} \colon |z-1/2|<1/2\}$. This allowed us to realise the Lyapunov exponent via the spectra of a family of trace-class operators on $H^2(D)$. In particular our results depended crucially on our ability to effectively estimate the singular values, or approximation numbers, of these operators, and also on our ability to estimate the gap between the largest and second-largest eigenvalues of one of these operators.

In order to obtain a version of Theorem \ref{main} for $d\times d$ matrices it is natural to proceed by analogy and consider the projective action on the open real simplex
\[\left\{\begin{pmatrix}x_1\\ \vdots \\x_d\end{pmatrix} \colon \sum_{i=1}^d x_i=1\text{ and }x_i>0\text{ for all }i=1,\ldots,d \right\}.\]
This now presents the problem of finding a complex extension of this simplex on which the real matrices induce a well-defined holomorphic action. One may show that such an extension  is provided by the corresponding projective slice in the natural complex extension of the positive cone:
\[\Omega:=\left\{\begin{pmatrix}z_1 \\ \vdots \\z_d\end{pmatrix} \colon \sum_{i=1}^dz_i=1\text{ and } \Re(z_i\overline{z_j})>0\text{ for all }1 \leq i,j \leq d\right\}.\]
We could then attempt to proceed by studying transfer operators on an appropriate space of holomorphic functions on $\Omega$. Since the structure and even definition of the space $H^2(\Omega)$ are sensitive to the shape of the boundary $\partial\Omega$, the Bergman space $A^2(\Omega)$ of holomorphic functions which are square-integrable over $\Omega$ may be more amenable. This would allow us to take advantage of existing estimates for the singular values of composition operators on Bergman spaces such as those in \cite{bj2}. 

The problem with this approach is that to the best of the authors' knowledge no fully effective estimates on the decay rates of these singular values exist in the literature. One may show that every positive $d\times d$ matrix induces a holomorphic transformation of $\Omega$ which maps $\Omega$ to a precompact subset of $\Omega$, which in particular does not approach to within some effectively-estimable distance $\varepsilon>0$ of the boundary of $\Omega$. Given a tuple of matrices $A_1,\ldots,A_k$ let us call the union of these images of $\Omega$ under the extensions of the projective actions of the matrices $A_i$ the \emph{image set} associated to those matrices. The estimates presented in \cite{bj2} allow the singular values of the associated composition operator on $A^2(\Omega)$ to be bounded in terms of $\varepsilon$ and in terms of a constant which relates to the minimum possible cardinality of a covering of the image set by ``strictly circled'' sets which are themselves subsets of $\Omega$. In order to effectively estimate the singular values (and hence the error in approximating the top Lyapunov exponent) by this method it would be necessary to understand the shape of $\Omega$ and the image set well enough to be able to estimate this covering constant effectively. Of course, in the case $d=2$ the set $\Omega$ is a Euclidean disc and therefore the image set can be contained in a single disc (which is a strictly circled set) making the estimation of the covering constant trivial. 
A potentially-viable alternative approach would be to search for an explicit orthogonal basis for $A^2(\Omega)$ or a related space which might allow the singular values to be estimated in a manner more directly analogous to the present work. In either event, higher-dimensional analogues of Theorem \ref{main} present an additional problem which is not present in the two-dimensional case.

\vspace{5mm}

\noindent \textbf{Acknowledgements.} Both authors were financially supported by the \emph{Leverhulme Trust} (Research Project Grant number RPG-2016-194).

\end{document}